\documentclass[11pt]{amsart}

\usepackage{mathrsfs}
\usepackage{amsmath,amssymb,enumerate,amsthm}
\usepackage{amsfonts,amssymb}
\usepackage{pinlabel}
\usepackage{latexsym,amsfonts,verbatim,mathrsfs,}
\usepackage{latexsym,graphics,textcomp,graphicx,subfigure,tikz}
\usepackage{tabularx,multirow,multicol,booktabs}
\usepackage{eucal,eufrak}
\usepackage{colonequals}
\usepackage{url}
\usepackage[colorlinks,linkcolor=red,anchorcolor=blue,
citecolor=green]{hyperref}

\theoremstyle{plain}
\newtheorem{theorem}{Theorem}[section]
\newtheorem{lemma}[theorem]{Lemma}
\newtheorem{corollary}[theorem]{Corollary}

\newtheorem{proposition}[theorem]{Proposition}

\theoremstyle{definition}
\newtheorem{definition}{Definition}
\newtheorem{remark}{Remark}
\newtheorem{question}{Question}

\newtheorem*{convention}{Convention}

\newcommand{\C}{\mathcal{C}}
\renewcommand{\H}{\mathcal{H}}

\newcommand{\R}{\mathbb R}
\newcommand{\Z}{\mathbb{Z}}
\newcommand{\ms}{\mathcal{S}}

\newcommand{\mT}{\mathcal{T}}

\newcommand{\gm}{\gamma}
\newcommand{\Gm}{\Gamma}

\newcommand{\bt}{\beta}

\newcommand{\dt}{\delta}
\newcommand{\Dt}{\Delta}

\newcommand{\omg}{\omega}

\renewcommand{\d}{\textup{d}}

\newcommand{\GL}{\textup{GL}}

\newcommand{\Aff}{\textup{Aff}}
\renewcommand{\i}{\mathcal{I}}
\newcommand{\J}{\mathcal{I}'}

\renewcommand{\P}{\mathbf{P}}
\newcommand{\inter}{\mathbf{int}} 
\title[Affine equivalence and saddle connection graphs ]{Affine equivalence and saddle connection graphs of ~half-translation~ surfaces}

\author{ Huiping Pan}

\address{Huiping Pan
\\Department of mathematics, Jinan University, 510632, Guangzhou, China}
\email{panhp@jnu.edu.cn}

\begin{document}

\begin{abstract}
 To every ~half-translation~ surface, we associate a saddle connection graph, which is a subgraph of the arc graph. We prove that every isomorphism between two saddle connection graphs is induced by an affine homeomorphism between the underlying ~half-translation~ surfaces. We also investigate the automorphism group of the saddle connection graph, and the corresponding quotient graph.
\end{abstract}
\maketitle
\begin{quote}
\small{
}
\end{quote}

\vskip 20pt
\section{Introduction}

\subsection{Arc complex}
For a compact oriented topological surface with marked points $(X,\Sigma)$, an arc $\underline{a}$ on $(X,\Sigma)$ is called \emph{properly embedded} if $\partial \underline{a}\subset\Sigma$ and the interior is disjoint from $\Sigma$. The \emph{arc complex} is a simplicial complex whose $k$-simplices  correspond to the set of $k+1$ isotopy classes of properly embedded non-trivial arcs which can be realized  pairwise disjointly outside the marked points. (In this paper, by an \textit{isotopy} we mean an isotopy relative to the marked points.) The arc complex is an important and useful tool for the study of mapping class group (\cite{Har1,Har2}). Masur and Schleimer (\cite{MS}) proved that the arc complex is $\dt$-hyperbolic (see also \cite{HH}). Later, Hesel-Przytycki-Webb (\cite{HPW}) proved that the arc complex is uniformly $7$-hyperbolic.

The mapping class group acts naturally on the arc complex. Irmak-McCarthy (\cite{IM}) proved that every injective simplicial map from the arc complex is induced by a self-homeomorphism. Based on this, they described completely the automorphism group of the arc complex.
(These results also hold  for the arc complex of non-orientable surface, \cite{Irm1,Irm2}.)

\subsection{Saddle connection graph} A ~half-translation~ surface is a pair $(X,\omg)$ where $X$ is a closed Riemann surface and $\omg$ is a meromorphic quadratic differential  on $X$ which contains no poles of order greater than one.  In parallel with the arc graph, which is the 1-skeleton of the arc complex for  a topological surface with marked points,  we can consider the saddle  connection graph for a half-translation~surface with marked points.

Let $(X,\omega;\Sigma)$ be a ~half-translation~ surface with marked points $\Sigma$ which contains all the zeros and poles of $\omg$. The meromorphic quadratic differential $\omg$ induces a singular flat metric $|\omg|$ on $X$ whose singular points are exactly the zeros and poles of $\omg$. A \emph{saddle connection} is an $|\omg|$-geodesic segment on $X\backslash\Sigma$ with endpoints in $\Sigma$. The \emph{saddle connection graph} of $(X,\omg;\Sigma)$, denoted by $\ms(X,\omg;\Sigma)$, is a graph such that the vertices are saddle connections and the edges are pairs of interiorly disjoint saddle connections.
This graph has infinite diameter (see Proposition \ref{prop:infdiam}). Moreover, it follows from an observation due to Minsky-Taylor (\cite{MT}) that the saddle connection graph is isometrically embedded into the arc graph. In particular, it is connected and $\dt$-hyperbolic (see Section 2).


\subsection{Statement of Results}
 The aim of this paper is to investigate isomorphisms between two  saddle connection graphs, and the automorphism group of the saddle connection graph. A homeomorphism between two ~half-translation~ surfaces with marked points is called \emph{affine} if it is affine outside the set of the marked points with respect to the coordinates defined by integrating one of the square roots of  the corresponding quadratic differential. Every affine homeomorphism induces an isomorphism between the corresponding saddle connection graphs. The main result of this paper shows that the converse is also true.

\begin{theorem}\label{thm:main1}
  Let $(X,\omg;\Sigma),(X',\omg';\Sigma')$ be two ~half-translation~ surfaces with marked points. Then every isomorphism $F: \ms(X,\omg;\Sigma)$ $\to\ms(X',\omg';\Sigma')$ is induced by an affine homeomorphism $f:(X,\omg;\Sigma)$ $\to(X',\omg';\Sigma')$.
\end{theorem}
\remark A direct consequence of Theorem \ref{thm:main1} is that every saddle connection graph determines a Teichm\"uller disk. In other words, there is a one-to-one correspondence between the set of isomorphism classes of saddle connection graphs and the set of Teichm\"uller disks.
\begin{theorem}\label{thm:main2}
  Let $(X,\omg;\Sigma)$ be a ~half-translation~ surface with marked points.
  \begin{enumerate}[(i)]
    \item  If $(X,\omg;\Sigma)$ is not a torus  with one marked point, then the automorphism group of $ \ms(X,\omg;\Sigma)$ is  isomorphic to the group of affine homeomorphisms of $(X,\omg;\Sigma)$.
    \item If $(X,\omg;\Sigma)$ is a torus  with one marked point, then the  automorphism group of $ \ms(X,\omg;\Sigma)$ is an index two subgroup of  the group of affine homeomorphisms of $(X,\omg;\Sigma)$.
  \end{enumerate}

\end{theorem}

\remark For a generic ~half-translation~ surface, the group of affine homeomorphisms is trivial. As a consequence, the corresponding saddle connection graph has no nontrivial automorphisms, which is different from the arc graph.

\begin{theorem}\label{thm:quotient}
  Let $(X,\omg;\Sigma)$ be a ~half-translation~ surface with marked points. Let $\mathcal{G}(X,\omg;\Sigma)$ be the quotient of $\ms(X,\omg;\Sigma)$ by its automorphism group.
\begin{enumerate}[(i)]
  \item If $(X,\omg;\Sigma)$ is a translation surface, then $\mathcal{G}(X,\omg;\Sigma)$ has infinitely many edges if and only if $(X,\omg;\Sigma)$ is not a torus with one marked point.
  \item If $(X,\omg;\Sigma)$ is a Veech surface, then $\mathcal{G}(X,\omg;\Sigma)$ has finitely many vertices.
\end{enumerate}
\end{theorem}
\remark We don't know whether the converse to the second statement in Theorem \ref{thm:quotient} is true or not (see Question \ref{q2}).

\vskip 5pt
\subsection{Related results}
To each Teichm\"uller disk, Smillie-Weiss (\cite{SW}) introduced the spine graph, which is a tree in the hyperbolic plane. They proved that the spine graph has compact quotient by the Veech group if and only if the Veech group is a lattice. The dual of the spine graph is a graph whose vertices are the directions of saddle connections and whose edges are pairs of directions which are the directions of the shortest saddle connections of some ~half-translation~ surface in the Teichm\"uller disk. Nguyen studied  the  graph of degenerate cylinders for translation~surfaces in genus two (\cite{Ngu1}) and the graph of periodic directions for translation~surface satisfying  the Veech dichotomy (\cite{Ngu2}). He proved that both of them are hyperbolic, and that every automorphism which comes from the mapping class group is induced by an affine self-homeomorphism. Moreover, based on the graph of periodic directions, Nguyen (\cite{Ngu2}) gave an algorithm to determine a coarse fundamental domain and a generating set for the Veech group of a Veech surface.

 Valentina Disarlo, Anja Randecker, and  Robert Tang  proved a similar result independently (\cite{DRT}) under the condition that all marked points are zeros or simple poles. The overall strategy of Disarlo-Randecker-Tang is similar to ours. But the ideas in proving the triangle preserving property are quite different. They prove the triangle preserving property via developing the ``triangle test", which is a combinatorial criterion that can detect the simplicies on the saddle connection complex that bound triangles on the underlying ~half-translation~ surfaces.

\vskip 5pt

\vskip 5pt
\subsection{Outline}
In Section 2, we  collect some basic properties of the saddle connection graph including the connectedness, hyperbolicity and infinite diameter. In Section \ref{sec:admissible:pentagon}-\ref{sec:affine}, we prove
 Theorem \ref{thm:main1}. The proof consists of three steps:
 \begin{quote}
    \begin{itemize}
  \item[Step 1.] (\S \ref{sec:admissible:pentagon}-\ref{sec:triangle}) Based on the existence of  many ``admissible pentagons" (see \S \ref{sec:admissible:pentagon}), we prove that the isomorphism $F$ preserves triangles (Theorem \ref{thm:triangle}).
  \item[Step 2.] (\S \ref{sec:homeo})  Fix a triangulation of $(X,\omg;\Sigma)$. The correspondence between triangles obtained in step 1 induces an affine map between triangles. We show that these affine maps have orientation consistency (Proposition \ref{prop:orientation:consist}), which allows us to glue these affine maps between triangles to obtain a homeomorphism from $(X,\omg;\Sigma)$ to $(X',\omg';\Sigma')$ (Theorem \ref{thm:homeo}). It then follows from the connectedness of the triangulation graph  that the isotopy class of the resulting homeomorphism is independent of the choices of triangulations (Proposition \ref{prop:consist:triangulation}).
  \item[Step 3.] (\S \ref{sec:affine}) We prove that the induced homeomorphism obtained in step 2 is isotopic to an affine homeomorphism (Proposition \ref{prop:affine}).
\end{itemize}
 \end{quote}
 Step three is a standard argument (see \cite{DLR,Ngu1,Ngu2}). The novel part of this paper is step one and step two.
In Section \ref{sec:affine}, we also prove Theorem \ref{thm:main2}. In  Section \ref{sec:quotient}, we prove Theorem \ref{thm:quotient}. In Section \ref{sec:questions}, we propose two questions.


\vskip 5pt
\subsection*{Acknowledgements} We  thank Duc-Manh Nguyen for pointing out a  mistake in the proof of Theorem \ref{thm:quotient} in an earlier version and for informing us the references \cite{BS,ILTC}. We thank Robert Tang for  comments for an earlier version. We thank Kasra Rafi for informing us the reference \cite{MT}. We also thank Lixin Liu and Weixu Su for useful discussions. We are grateful to the referees for careful reading of the manuscript, with many  corrections and  comments. This work is supported by NSFC 11901241.

\section{Preliminaries}

\subsection{Arc graph}
An arc $\underline{a}$ on $(X,\Sigma)$ is called \emph{properly embedded} if $\partial \underline{a}\subset\Sigma$ and the interior is disjoint from $\Sigma$. A properly embedded arc is said to be \textit{trivial} if it is isotopic to a  single point (i.e. the endpoint). Otherwise, it is said to be \textit{non-trivial}.
 The \emph{arc graph} of $(X,\Sigma)$, denoted by $\mathcal{A}(X,\Sigma)$, is a graph such that the vertices $\mathcal{A}^0(X,\Sigma)$ are isotopy classes of properly embedded non-trivial arcs on $(X,\Sigma)$, the edges $\mathcal{A}^1(X,\Sigma)$ are pairs of isotopy classes of properly embedded non-trivial arcs which can be realized interiorly disjoint.
(In this paper, by an \textit{isotopy} we mean an isotopy relative to the marked points.)

Let each edge in $\mathcal{A}(X,\Sigma)$ be of length one. This induces a metric $\d_{\mathcal{A}}$ on $\mathcal{A}(X,\Sigma)$. Recall that a geodesic metric space $(\mathbb{X},\d)$ is called \emph{$\dt$-hyperbolic} if for every geodesic triangle $[xy]\cup[yz]\cup[zx]$, each geodesic in $\{[xy],[yz],[zx]\}$ is contained in the $\dt$-neighbourhood of the union of the other two.

\begin{theorem}[\cite{MS,HH,HPW}]\label{thm:hyp}
  The arc graph  $(\mathcal{A}(X,\Sigma), \d_{\mathcal{A}})$ is $\dt$-hyperbolic.
\end{theorem}
\remark In fact, Hesel-Przytycki-Webb (\cite{HPW}) show that $(\mathcal{A}(X,\Sigma),\d_{\mathcal{A}})$ is 7-hyperbolic.

\subsection{~Half-translation~ surfaces}

\begin{definition}[half-translation~ surface]
   A \emph{~half-translation~ surface} is a pair $(X,\omg)$ where $X$ is a closed  Riemann surface  and $\omg$ is a meromorphic quadratic differential on $X$ which contains no poles of order greater than one.
    An \emph{~half-translation~ surface with marked points} is a triple $(X,\omg;\Sigma)$ such that $(X,\omg)$ is a half-translation~surface, and  $\Sigma$ is a finite subset of $X$ which contains the zeros and  poles of $\omg$.
\end{definition}

The meromorphic quadratic  differential $\omg$ induces a singular flat metric $|\omg|$ on $X$, which is flat outside the zeroes and poles of $\omg$. The zeros and poles of $\omg$ are called \textit{singular points}. The cone angle at each singular point of $\omg$ is a multiple of $\pi$.

\begin{definition}[Saddle connection]
  A \emph{saddle connection} is an $|\omg|$-geodesic segment on $X\backslash\Sigma$ with endpoints in $\Sigma$.
\end{definition}

For an oriented  saddle connection ${\alpha}$, the integral $\int_{{\alpha}}\sqrt{\omg}$ is called the \emph{holonomy} of ${\alpha}$, where $\sqrt{\omg}$ is one of the square roots of $\omg$. Since there is no canonical choice of square roots of $\omg$, we see that the holonomies of (oriented)  saddle connections belong to $\mathbb{R}^2/\{\pm1\}$.
\begin{proposition}
  Let $(X,\omg;\Sigma)$ be a ~half-translation~ surface with marked points.
  \begin{itemize}
    \item   The set of holonomies of saddle connections on $(X,\omg;\Sigma)$ is a discrete subset of $\R^2/\{\pm1\}$.
    \item  The set of directions of saddle connections on $(X,\omg;\Sigma)$ is a dense subset of $
        \mathbb{S}^1/\{\pm1\}$.
  \end{itemize}
\end{proposition}
\remark The first statement can be found in \cite[Proposition 1]{HS}. The second statement is \cite[Theorem 2]{Mas}.

\begin{definition}[Cylinder]\label{def:semi-cylinder}
   A \emph{cylinder} on $(X,\omg;\Sigma)$ is an open subset disjoint from $\Sigma$,  which is isometric to $(\R/c\Z)\times (0,h)$ with $c,h\in\R_{>0}$ and not properly contained in any other subset with the same property. A boundary component of a cylinder is called \textit{simple} if it consists of only one saddle connection.
   A cylinder is called \textit{semisimple} if it contains at least one simple boundary component.
   A cylinder  is called
   \textit{simple} if both of its boundary components are simple.
\end{definition}

  By a simple closed curve on $(X,\omg;\Sigma)$, we  mean   a simple closed curve on $X\backslash \Sigma$.
\begin{definition}[Cylinder curve]
  A simple closed curve on $(X,\omg;\Sigma)$ is called a \textit{cylinder curve}  if it is isotopic to a core curve of some cylinder on $(X,\omg;\Sigma)$.
\end{definition}

 For our purpose, we need the following theorem which is stated as Lemma 22 in \cite{DLR}.
\begin{theorem}[\cite{DLR}]\label{thm:DLR}
  Let $f:(X,\omg;\Sigma)\to (X',\omg';\Sigma')$ be a homeomorphism preserving marked points, such that for each simple closed curve $\alpha$, $\alpha$ is a cylinder curve on $(X,\omg;\Sigma)$ if and only if $f(\alpha)$ is a cylinder curve on $(X',\omg';\Sigma')$. Then $f$ is isotopic to an affine homeomorphism.
\end{theorem}

\subsection{Saddle connection graph}
\begin{definition}
  The \emph{saddle connection graph} of $(X,\omg;\Sigma)$, denoted by $\ms(X,\omg;\Sigma)$, is a graph such that the vertices $\ms^0(X,\omg;\Sigma)$ are saddle connections and the edges $\ms^1(X,\omg;\Sigma)$ are pairs of disjoint saddle connections.
\end{definition}

\begin{convention}
  In this paper, whenever we mention the intersection between two saddle connections, we mean the intersection in the interior.
\end{convention}

Notice that the saddle connection graph $\ms(X,\omg;\Sigma)$ is a subgraph of $\mathcal{A}(X,\Sigma)$. For any topological arc $\underline{a}\in \mathcal{A}^0(X,\omg)$, the $|\omg|$-geodesic representative consists of several saddle connections. It follows from \cite[Theorem 1.4]{MT} that $\ms(X,\omg;\Sigma)$ is a  geodesically connected subset of $\mathcal{A}(X,\Sigma)$, in the sense that any two points in $\ms(X,\omg;\Sigma)$ are connected by a geodesic of $\mathcal{A}(X,\Sigma)$ that lies in $\ms(X,\omg;\Sigma)$. This implies that $\ms(X,\omg;\Sigma)$ is \textit{connected} and \textit{$\delta$-hyperbolic.}

\begin{definition}[Triangulation]
  A triangulation $\Gm=\{\gm_i\}$ of $(X,\omg;\Sigma)$ is a set of saddle connections $\{\gm_1,\gm_2,\cdots,\gm_{6g-6+3p}\}$, where $p=|\Sigma|$ and $g$ is the genus of $X$, such that
  \begin{itemize}
    \item for any $i\neq j$, $\gm_i$ and $\gm_j$ are disjoint;
    \item any saddle connection $\gm\notin\Gm$ intersects $\gm_i$  for some $\gm_i\in\Gm$.
  \end{itemize}
\end{definition}
\begin{lemma}\label{lem:triangulation:correspondence}
  Let $F:\ms(X,\omg;\Sigma)\to \ms(X',\omg';\Sigma')$ be an isomorphism between two saddle connection graphs. Let  $\Gm:=\{\gm_i\}$ be a triangulation  of $(X,\omg;\Sigma)$.
        Then $F(\Gm):=\{F(\gm):\gm\in\Gm\}$ is a triangulation of $(X',\omg';\Sigma)$.
\end{lemma}
\begin{proof}
 By definition, we see that for any pair of different saddle connections $\gm_i,\gm_j\in\Gm$, $F(\gm_i)$ and $F(\gm_j)$ are disjoint. If $F(\Gm)$ is not a triangulation of $(X',\omg';\Sigma')$, then there exists a saddle connection $\gm' \notin F(\Gm)$ of $(X',\omg';\Sigma')$ which is disjoint from  $F(\gm_i)$ for all $\gm_i\in F(\Gm)$. Correspondingly, $F^{-1}(\gm')\notin \Gm$ is a saddle connection  of $(X,\omg;\Sigma)$ which is disjoint from all $\gm_i\in \Gm$. This is a contradiction which proves the lemma!

\end{proof}


\subsection{$\GL(2,\R)$ action}
Let $\kappa=(k_1,\cdots,k_n)\in\{-1,0,1,\cdots,4g-4\}^n$ be a partition of $4g-4$, i.e. $k_1+k_2\cdots +k_n=4g-4$. Denote by  $\mathcal{Q}(\kappa)$  the moduli space of meromorphic quadratic differentials with marked points $(X,\omg;\Sigma)$ which contain $n$ zeros of multiplicities $(k_1,\cdots k_n)$, where a zero of multiplicity $-1$ is a pole of order one and a zero of multiplicity zero is a regular marked point. There is a natural $\GL(2,\R)$ action on $\mathcal{Q}(\kappa)$, which acts on $(X,\omega;\Sigma)\in \mathcal{Q}(\kappa)$ by post-composition with the atlas maps of $(X,\omega;\Sigma)$ (see \cite[\S 1.8]{MT}). More precisely, let $\{(U_i,\phi_i),i\in I\}$ be an atlas of $(X,\omega;\Sigma)$, covering $X$ except the marked points of $\omega$. Then for $a\in\GL(2,\R)$, $a\cdot(X,\omega;\Sigma)$  is defined by the atlas  $\{(U_i,a\circ\phi_i),i\in I\}$. Therefore, the identity map between $(X,\omg;\Sigma)$ and $a\cdot (X,\omg;\Sigma)$ is isotopic to an affine homeomorphism, which induces an isomorphism  between $\ms(X,\omg;\Sigma)$ and $\ms(a\cdot (X,\omg;\Sigma))$.

For each $\theta\in \mathbb{S}^1/\{\pm1\}$, the straight lines in that direction induce a directional flow on $(X,\omg;\Sigma)$. By  \cite[Theorem 1]{KMS}, for  almost every $\theta\in\mathbb{S}^1/\{\pm1\}$, the induced directional flow is \textit{uniquely ergodic} on $(X,\omg;\Sigma)$.

\subsection{Infinite diameter.}
\begin{proposition}\label{prop:infdiam}
  Let $(X,\omg;\Sigma)$ be a ~half-translation~ surface with marked points. The saddle connection graph $\ms(X,\omg;\Sigma)$ has infinite diameter.
\end{proposition}
\begin{proof}
  We use an argument of F. Luo as explained in \cite[\S4.3]{MM}.
  Suppose to the contrary that the diameter of $\ms(X,\omg;\Sigma)$ is finite.

   By  \cite[Theorem 1]{KMS}, for  almost every $\theta\in\mathbb{S}^1/\{\pm1\}$, the induced directional flow is \text{uniquely ergodic} on $(X,\omg;\Sigma)$.
  After rotation, we may assume that  the horizontal flow of $(X,\omg;\Sigma)$ is uniquely ergodic. Fix a non-horizontal saddle connection $\alpha$. Let $\{\bt_i\}_{i\geq1}$ be a sequence of saddle connections, such that the limit of the directions is horizontal. This implies that for every non-horizontal segment $I$, there exists $T>0$, such that $\bt_i$  intersects $I$  for all $i>T$.

  By assumption, the diameter of $\ms(X,\omg;\Sigma)$ is finite, there is a subsequence, still denoted by $\{\bt_i\}_{i\geq1}$ for convenience, such that
  $\d_\ms(\alpha,\bt_i)=N$ for some $N<\infty$. For each $i$, consider a geodesic $\bt^0_i=\alpha, \bt^1_i,\cdots, \bt^{N-1}_i,\bt^N_i=\bt_i$ in the saddle connection graph from $\alpha$ to $\bt_i$.
  On the other hand,
  for every non-horizontal segment $I$, $\bt_i$  intersects $I$ for all large enough $i$. Therefore, the limit of the directions of $\bt_i^{N-1}$ has to be horizontal. Otherwise, $\bt_i=\bt^N_i$ would intersect $\bt_i^{N-1}$ for large enough $i$. Inductively, we obtain that  the limit of the directions of $\bt_i^{1}$ has to be horizontal. As a consequence, $\bt^1_i$ would intersect the non-horizontal saddle connection $\alpha=\bt^0_i$ for large enough $i$, which contradicts that $\d_\ms(\alpha,\bt^1_i)=1$.

\end{proof}


\section{Admissible polygons}\label{sec:admissible:pentagon}

Let $(X,\omg;\Sigma)$ be a half-translation surface with marked points. The goal of this section is to find ``admissible polygons" on $(X,\omg;\Sigma)$.
Let us start with the following definition.

\begin{definition}[admissible map]\label{def:admap}
  Let $\mathbf Q\subset\R^2$ be a polygon. A continuous map $\i:\mathbf Q\to(X,\omg;\Sigma)$ is called \textit{admissible} if
  \begin{enumerate}
    \item it is a local ~half-translation, i.e. the derivative  is $\pm\begin{pmatrix}
                      1 & 0 \\
                      0 & 1
                    \end{pmatrix} $;
    \item it is an embedding in the interior of $\mathbf Q$;
    \item $\i(P)$ is a marked point if and only if $P$ is a vertex of $\mathbf{Q}$.
  \end{enumerate}
  A polygon   is called \textit{admissible} if it admits an  admissible map into $(X,\omg;\Sigma)$.
\end{definition}
In particular, an admissible map is a local isometry sending  sides and diagonals (interior to $\mathbf{Q}$) to  saddle connections.

We will  use the notation $(A_1A_2\cdots A_m)$ to represent a polygon  $\mathbf{Q}\subset\R^2$ with counterclockwise labeled vertices $A_1,A_2,\cdots A_m$.

\begin{lemma}\label{lem:admissible:boundary:disjoint}
  Let $\i:(A_1A_2\cdots A_m)\to (X,\omg;\Sigma)$ be an admissible map, then for any pair of sides $\overline{A_iA_{i+1}}$ and $\overline{A_jA_{j+1}}$, $\i(\overline{A_iA_{i+1}})$ and $\i(\overline{A_jA_{j+1}})$ are either equal or disjoint.
\end{lemma}
\begin{proof}
  It follows directly from the definition of admissible maps.
\end{proof}

\begin{definition}
  A convex polygon  on $\R^2$ is called \textit{strictly convex} if it is strictly convex at each vertex, i.e. the interior angle at each vertex is strictly less than $\pi$.
\end{definition}
 \begin{lemma}\label{lem:translation}
   Let $\P$ be a convex polygon in $\R^2$. Let
   $\i:\P\to(X,\omg;\Sigma)$ be a local half-translation such that $\i(P)\in\Sigma$ only if $P$ is a vertex. Let $Q_1,Q_2 \in\P$ be such that $\inter(\overline{Q_1Q_2})$ contains no vertices of $\P$. If $\i(Q_1)=\i(Q_2)$, then the holonomy corresponding to the closed curve $\i(\overline{Q_1Q_2})$ is a translation.
 \end{lemma}
 \begin{proof}

   Suppose to the contrary that the holonomy is a strict half-translation, i.e. $Q_1$ and $Q_2$ are identified by  a strict half-translation. Let $Q\in\P$ be the midpoint of $Q_1,Q_2$. Then $\i(Q)$ would be a marked point of cone angle $\pi$. Hence, $Q$ is a vertex of $\P$.  This is a contradiction which proves the lemma.
 \end{proof}

\begin{lemma}[Admissible extension I]\label{lem:extension1}
 Let $(A_1A_2\cdots A_m)\subset\R^2$ be a polygon without self-intersections.
  Let $\i:(A_1A_2\cdots A_{n})\to (X,\omg;\Sigma)$ and $\tilde{\i}:(A_1A_{n}\cdots A_m)\to (X,\omg;\Sigma)$ be two admissible maps, $3\leq n\leq m-1$, such that
  \begin{itemize}
    \item $\i(\overline{A_1A_{n}})=\tilde{\i}(\overline{A_1A_{n}})$;
     \item $\i(\overline{A_1A_n})\neq \i(\overline{A_iA_{i+1}})$ for all $1\leq i\leq n-1$ and $\tilde{\i}(\overline{A_1A_n})\neq \tilde{\i}(\overline{A_jA_{j+1}})$ for all $n\leq j\leq m$;
    \item for  $\forall 1\leq i\leq n-1$,  $ \forall n\leq j\leq m$, $\i(\overline{A_iA_{i+1}})$ and $\tilde{\i}(\overline{A_jA_{j+1}})$ are either equal or disjoint.
  \end{itemize}
  Then the map $\hat{\i}:(A_1A_2\cdots A_m)\to (X,\omg;\Sigma)$ defined by
  $$\hat{\i}|_{(A_1A_2\cdots A_{n})}=\i,~ \hat{\i}|_{(A_1A_{n}\cdots A_m)}=\tilde{\i}$$
  is admissible.
\end{lemma}
\begin{proof}
  To show that $\hat{\i}$ is admissible, it suffices to show that $\hat{\i}$ is an embedding in the interior. Suppose to the contrary that there exist $Q_1,Q_2\in\inter(A_1A_2\cdots A_m)$ such that $\hat{\i}(Q_1)=\hat{\i}(Q_2)$. Moreover, we may assume that $Q_2\in \inter(A_1A_{n}\cdots A_m)\cup \inter(\overline{A_1A_{n}})$ and  $Q_1\in \inter(A_1A_2\cdots A_{n})\cup \inter(\overline{A_1A_{n}})$.
  Since $(X,\omg;\Sigma)$ is a half-translation surface, the identification between $Q_1$ and $Q_2$ by $\hat{\i}$ is a  half-translation of $\R^2$. Let $\zeta$ be the induced half-translation  sending $Q_2$ to $Q_1$. Let $(A_1'A_{n}' \cdots A_m')$ be the image of $(A_1A_{n} \cdots A_m)$ by $\zeta$. Then
  $$ {\i}(P)= \tilde\i(\zeta^{-1}(P)), ~\forall P\in (A_1A_{n} \cdots A_m)\cap (A_1'A_{n}'\cdots A_m').$$
  Notice that
  $$Q_1\in \inter(A_1A_2\cdots A_{n})\cap \inter(A_1'A_{n}'\cdots A_m').$$
  Combined with the third assumption in the setting, this implies that
   $$(A_1'A_{n}'\cdots A_m')\subset (A_1A_2\cdots A_{n}) \text{ or  }
    (A_1A_2\cdots A_{n})\subset(A_1'A_{n}'\cdots A_m').$$
  Without loss of generality, we may assume that $(A_1'A_{n}'\cdots A_m')\subset (A_1A_2\cdots A_{n})$. Then  $\overline{A_1'A_n'}=\overline{A_iA_j}$ for some side or diagonal of $(A_1A_2\cdots A_n)$ (because both $(A_1'A_{n}'\cdots A_m')$ and $(A_1A_2\cdots A_{n})$ are admissible polygons of $(X,\omg;\Sigma)$). Hence,
  $\i(\overline{A_1A_n})=\tilde{\i}(\overline{A_1A_n})=
  \i(\overline{A_1'A_n'})=\i(\overline{A_iA_j})$. Since $\i$ is admissible, it follows that $\overline{A_iA_j}$ can not be a diagonal. Then $\overline{A_iA_j}$ must be a side. It then follows from the second assumption in the setting that $\overline{A_iA_j}=\overline{A_1A_n}$. Therefore, $\zeta(\overline{A_1A_n})=\overline{A_1'A_n'}=\overline{A_1A_n}$. Consequently,  $\zeta$ is a strict half-translation, and the midpoint of $\overline{A_1A_n}$ is marked point of cone angle $\pi$. This contradicts to the assumption that $\i$ is admissible.
\end{proof}

 Let $\H:=\{(x,y)\in\R^2: -r<y<r\}$ be a horizontal strip with upper boundary $\partial_+\H$ and lower boundary $\partial_-\H$. Let $A_i=(a_i,b_i)\in\overline{\H}$, $i=1,2,\cdots, n$,  such that
 \begin{itemize}
   \item  $b_n=-r\leq b_{n-1}\leq\cdots\leq b_2\leq b_1=r$;
   \item $a_1=a_n=0$, $a_i<0$ for all $i=2,\cdots,n-1$.
 \end{itemize}
 \begin{lemma}[Admissible extension II]\label{lem:emdedding:strip}
   Let $\H, A_1,~A_2\cdots A_n$ be as above. Let $\i:(A_1A_2\cdots A_n)\to (X,\omg;\Sigma)$ be an admissible map such that $\i(\overline{A_1A_n})\neq \i(\overline{A_iA_{i+1}})$ for any $1\leq i\leq n-1$. Then there exists $A_{n+1}\in\overline{\H}$ on the right side of $\overline{A_1A_n}$, such that $\i$ can be admissibly extend to $(A_1A_2\cdots A_nA_{n+1})$. Moreover, if the image of $(A_1A_2\cdots A_n)$ by $\i$ is not contained in any horizontal cylinder, then we have $A_{n+1}\in\H$.
 \end{lemma}

\begin{figure}[b]
  \centering
 \begin{tikzpicture}[scale=1.1]
   \draw[dashed] (0,0)--(5,0)--(5,-3)--(0,-3)--cycle;
  \path[fill,black] (3.5,0) coordinate(a1) circle(0.05) node[above]{\tiny{$A_1$}}
   (1,-1) coordinate(a3) circle(0.05) node[above]{\tiny{$A_{k-1}$}}
   (2.6,-2.5) coordinate(a4) circle(0.05) node[below]{\tiny{$A_k$}}
   (3.5,-3) coordinate(a6) circle(0.05) node[below]{\tiny{$A_n$}}
   (5,-0.5) coordinate(a7) circle(0.05) node[right]{\tiny{$A_{n+1}$}}

    (3.5-2.6,0-1.5) coordinate(a12) circle(0.05) node[left]{\tiny{$\hat A_1$}}
   (3.5-2.6,-3-1.5) coordinate(a62) circle(0.05) node[left]{\tiny{$\hat A_n$}}
    (5.0-2.6,-0.5-1.5) coordinate(a72) circle(0.05) node[right]{\tiny{$\hat A_{n+1}$}};
   \draw
    (5,0) coordinate(p3) circle(0.05) node[above]{\tiny{$P_a$}}
    (5,-3) coordinate(q3) circle(0.05) node[below]{\tiny{$Q_a$}}
     (5-2.6,0-1.5) coordinate(p32) circle(0.05) node[right]{\tiny{$\hat P_a$}}
    (5-2.6,-3-1.5) coordinate(q32) circle(0.05) node[right]{\tiny{$\hat Q_a$}}
   (3.5,-1)coordinate(u) circle(0.05) node[right]{\tiny{$U$}}
    (3.5,-2.5)coordinate(v) circle(0.05) node[right]{\tiny{$V$}}
      (5-2.6-0.5,0-1.85)  circle(0.05)node[below]{\tiny{$\hat{S}$}}
       (5-0.5,-0.35)circle(0.05) node[below]{\tiny{$S$}}
       (5-2.6,-5)node[right]{\tiny{(a)}};
    \draw (a1)--(a6)--(a7)--cycle  (a3)--(a4);
    \draw[dashed] (a12)--(p32)--(q32)--(a62)--cycle  (a12)--(a72)--(a62)
    (a3)--(u)  (a4)--(v);

    \draw[dashed] (0+6,0)--(5+6,0)--(5+6,-3)--(0+6,-3)--cycle;
  \path[fill,black] (3.5+6,0) coordinate(a1) circle(0.05) node[above]{\tiny{$A_1$}}
   (1+6+1.7,-0.5) coordinate(a3) circle(0.05) (1+6+1.9,-0.5)node[below]{\tiny{$A_{k-1}$}}
   (1.2+6,-2.5) coordinate(a4) circle(0.05) node[below]{\tiny{$A_k$}}
   (3.5+6,-3) coordinate(a6) circle(0.05) node[below]{\tiny{$A_n$}}
   (5+6,-2.5) coordinate(a7) circle(0.05) node[right]{\tiny{$A_{n+1}$}}

    (3.5-2.6+6,0-1.5+2.5) coordinate(a12) circle(0.05) node[left]{\tiny{$\hat A_1$}}
   (3.5-2.6+6,-3-1.5+2.5) coordinate(a62) circle(0.05) node[left]{\tiny{$\hat A_n$}}
    (5.0-2.6+6,-1.5) coordinate(a72) circle(0.05) node[right]{\tiny{$\hat A_{n+1}$}};
   \draw
   (5+6,0) coordinate(p3) circle(0.05) node[above]{\tiny{$P_a$}}
    (5+6,-3) coordinate(q3) circle(0.05) node[below]{\tiny{$Q_a$}}
     (5-2.6+6,0-1.5+2.5) coordinate(p32) circle(0.05) node[right]{\tiny{$\hat P_a$}}
    (5-2.6+6,-3-1.5+2.5) coordinate(q32) circle(0.05) node[right]{\tiny{$\hat Q_a$}}
   (3.5+6,-0.5)coordinate(u) circle(0.05) node[right]{\tiny{$U$}}
    (3.5+6,-2.5)coordinate(v) circle(0.05) node[right]{\tiny{$V$}}
     (2.2+6,-1.2)  circle(0.05)node[left]{\tiny{$\hat{S}$}}
       (5+6-0.2,-2.5+0.3)circle(0.05) node[left]{\tiny{$S$}}
       (5-2.6+6,-5)node[right]{\tiny{(b)}};
    \draw (a1)--(a6)--(a7)--cycle  (a3)--(a4);
    \draw[dashed] (a12)--(p32)--(q32)--(a62)--cycle  (a12)--(a72)--(a62)
    (a3)--(u)  (a4)--(v);
 \end{tikzpicture}
  \caption{Admissible extension of polygons.}\label{Fig:embedding:strip}
\end{figure}
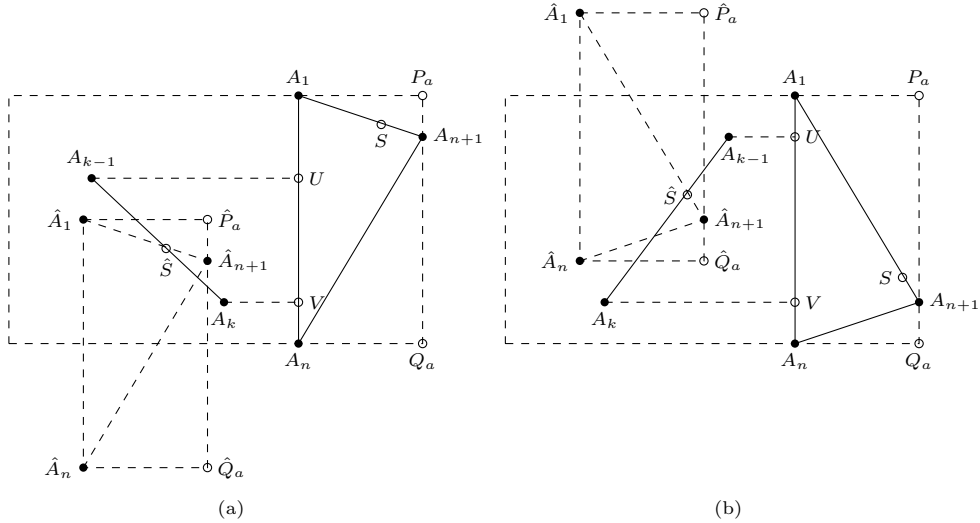
 \begin{proof}
   Let $P_t\in\partial_+\H$ and $Q_t\in\partial_-\H$ such that $\overrightarrow{A_1P_t}=\overrightarrow{A_nQ_t}=t(1,0)$, where $t>0$. For $t>0$ small, by following the horizontal lines  on $(X,\omg;\Sigma)$ through points in $\i(\overline{A_1A_n})$, we can extend $\i$  to  the polygon $(A_1A_2\cdots A_n Q_t P_t)$ which remains to be  a local half-translation. In fact, the extension fails only if some interior point of $\overline{P_tQ_t}$ is mapped to a marked point by $\i$. The assumptions that $\i:(A_1A_2\cdots A_n)\to (X,\omg;\Sigma)$ is an admissible map and that $\i(\overline{A_1A_n})\neq \i(\overline{A_iA_{i+1}})$ for any $1\leq i\leq n-1$ imply that for small $t>0$, the extension  $\i:(A_1A_2\cdots A_n Q_t P_t)\to(X,\omg;\Sigma)$ is an embedding in the interior.

   Now, let us consider the horizontal rays $\{L_P: P\in \i(\overline{A_1A_n})\}$ on $(X,\omg;\Sigma)$ which emanate from the  points of $\i(\overline{A_1A_n})$ and which flow away from $(A_1A_2$ $\cdots A_n)$.  Since $(X,\omg;\Sigma)$ is of finite area and  $\{\i(A_1),\i(A_n)\}\subset\Sigma$,  it follows that some of such rays would meet  marked points during the flowing process.  Let $L_T$, $T\in \i(\overline{A_1A_n})$ be a ray which meet marked points first (if there are more than one candidates, we choose the one closer to $A_1$).  Denote by $a$ the distance it travels from $T$ to the first marked point.  Let $A_{n+1}=(a_{n+1},b_{n+1})\in\overline{P_aQ_a}$ be such that $\i(A_{n+1})$ is the marked point met by $L_T$  (see Figure \ref{Fig:embedding:strip}). Then $a_{n+1}=a>0$ and $b_{n+1}\geq-r$.

    We claim that $i:(A_1A_nA_{n+1})\to (X,\omg;\Sigma)$ is admissible.  Suppose to the contrary that there exist $M=(x,y),~\hat{M}=(\hat{x},\hat{y})\in\mathbf{int}(A_1 A_nA_{n+1})$ such that $\i(M)=\i(\hat{M})$. By Lemma \ref{lem:translation}, the identification between $\hat{M}$ and $M$ by $\i$ is a translation.  Without loss of generality, we may assume that $x\geq \hat{x}$. Furthermore, we assume that $y\geq \hat{y}$. (The proof for the case $y\leq \hat{y}$ is similar.) Let $\mathbf{k}$ and $k$ be respectively the slopes of $\overline{A_{n}A_{n+1}}$ and $\overline{\hat{M}M}$. Then $\mathbf{k},k\geq 0$. If $k\leq \mathbf{k}$, then $\exists \tilde{M}\in\inter(A_1A_{n}A_{n+1})\backslash\{A_1,A_n,A_{n+1}\}$ such that $\overrightarrow{\tilde{M}A_{n+1}}=\overrightarrow{\hat{M}M}$. Then $\i(\tilde{M})=\i(A_{n+1})\in \Sigma$, which contradicts to the construction of $A_{n+1}$. If $k> \mathbf{k}$, then  $\exists \tilde{M}\in\inter(A_1A_{n}A_{n+1})\backslash\{A_1,A_n,A_{n+1}\}$ such that $\overrightarrow{A_n\tilde{M}}=\overrightarrow{\hat{M}M}$. Then $\i(\tilde{M})=\i(A_{n})\in \Sigma$, which also contradicts to the construction of $A_{n+1}$.

    \vskip5pt
    By Lemma \ref{lem:extension1},
    to show that $\i:(A_1A_2\cdots A_{n+1})\to(X,\omg;\Sigma)$ is admissible, it suffices to show that  for  $\forall 1\leq i\leq n-1$,  $ \forall n\leq j\leq n+1$, $\i(\overline{A_iA_{i+1}})$ and ${\i}(\overline{A_jA_{j+1}})$ are either equal or disjoint. In the following, we shall consider $\overline{A_{n+1}A_{1}}$. The proof for $\overline{A_nA_{n+1}}$ is similar.

    Suppose to the contrary that there exists $2\leq k\leq n$ such that $\i(\overline{A_{n+1}A_1})$ intersects $\i(\overline{A_{k-1}A_{k}})$ but  $\i(\overline{A_{n+1}A_1})\neq\i(\overline{A_{k-1}A_{k}})$. Let $S\in\inter(\overline{A_{n+1}A_1})$ and $\hat{S}\in\inter(\overline{A_{k-1}A_k})$ such that $\i(S)=\i(\hat{S})$. Since $(X,\omg;\Sigma)$ is a half-translation surface, the identification between $S$ and $\hat{S}$ by $\i$ has two cases:  translation or strict half-translation.

    \vskip 5pt
    \begin{itemize}
      \item \textbf{Case 1}. The identification between $S$ and $\hat{S}$ is a translation. Let $\Theta:\R^2\to\R^2$ be the translation such that $\Theta(S)=\hat{S}$.
    Let $(\hat{A}_1\hat{P}_a\hat{A}_{n+1}\hat{Q}_a\hat{A}_n)$ be the image of $(A_1P_aA_{n+1}Q_aA_n)$  under the translation $\Theta$. In particular,   $(\hat{A}_1\hat{P}_a\hat{Q}_a\hat{A}_n)$ is a parallelogram with $\overline{\hat{A}_1\hat{P}_a}$ horizontal and $\overline{\hat{P}_a\hat{Q}_a}$ vertical.
    Moreover,
    \begin{equation}\label{eq:hat:S}
      \inter({\overline{\hat{A}_1\hat{A}_{n+1}}})\cap \inter({ \overline{A_{k-1}A_k}})=\{\hat{S}\}.
    \end{equation}

    By the construction of $A_{n+1}$, we see that
    \begin{equation}\label{eq:k:k+1}
      A_{k-1},A_k\notin
    \inter(\hat{A}_1\hat{P}_a\hat{A}_{n+1}\hat{Q}_a\hat{A}_n)\cup
    \inter(\overline{\hat{A}_1\hat{P}_a})\cup\inter(\overline{\hat{A}_n\hat{Q}_a}).
    \end{equation}

    Since $\i(\overline{A_1A_n})\neq\i(\overline{A_{k-1}A_k})$ and $\i:(A_1A_2\cdots A_n)\to(X,\omg;\Sigma)$ is admissible, it follows that $\overline{A_{k-1}A_k}$ and $\overline{\hat{A}_1\hat{A}_{n}}$ are disjoint.
   Combining  (\ref{eq:hat:S}), (\ref{eq:k:k+1}), and the assumption about the $y$-coordinates of $A_{k-1}$ and $A_k$, we see that
     $\overline{A_{k-1}A_{k}}$ intersects both $\overline{\hat{A}_1\hat{P}_a}$ and $\overline{\hat{P}_a\hat{Q}_a}$ (see Figure \ref{Fig:embedding:strip}(a)), or
    it intersects both $\overline{\hat{A}_n\hat{Q}_a}$ and $\overline{\hat{P}_a\hat{Q}_a}$ (see Figure \ref{Fig:embedding:strip}(b)). Without loss of generality, we assume that
    $\overline{A_{k-1}A_{k}}$ intersects both $\overline{\hat{A}_1\hat{P}_a}$ and $\overline{\hat{P}_a\hat{Q}_a}$.

     Let $U,V\in\overline{A_1A_n}$ such that $\overline{A_{k-1}U}$ and $\overline{A_{k}V}$ are horizontal.
     Then $$\hat{A}_{n+1}\in
      \inter(A_{k-1}A_{k}VU)\subset\inter(A_1A_2\cdots A_n),$$
       which contradicts with the assumption that $(A_1A_2\cdots A_n)$ is an admissible polygon.

     \vskip5pt
      \item \textbf{Case 2}. The identification between $S$ and $\hat{S}$ is a strict half-translation. Let $\widetilde{\Theta}$ be  the half-translation of $\R^2$ such that $\widetilde{\Theta}(S)=\hat{S}$.  Let $(\tilde{A}_1\tilde{P}_a\tilde{A}_{n+1}\tilde{Q}_a\tilde{A}_{n})$ be the image of $(A_1P_aA_{n+1}Q_aA_n)$ by $\widetilde{\Theta}$. Then
    $$ (\tilde{A}_1\tilde{P}_a\tilde{A}_{n+1}\tilde{Q}_q\tilde{A}_{n})\cap\overline{\H}
    \subset\{(x,y)\in\overline{\H}:x\geq0\}. $$
    (Otherwise, some ray in $\{L_P:P\in\overline{A_1A_n}\}$ would meet the marked point $\i(A_n)$ before $L_T$ meet $\i(A_{n+1})$.) In particular, we have $\hat{S}\in (A_1A_{n+1}A_n)$.  It then follows from Lemma \ref{lem:translation} that the identification between $S$ and $\hat{S}$ by $\i$ is a translation, which contradicts to the assumption of this case.
    \end{itemize}

   \vskip5pt

    If in addition, $\i(\inter(A_1A_2\cdots A_n))$ is not contained in any horizontal cylinder,  we consider the family of rays  $\{L_P: P\in\mathbf{int}(\i(\overline{A_1A_n}))\}$  instead of $\{L_P:P\in\overline{A_1A_n})\}$.  It then follows that $A_{n+1}\in\H$.
 \end{proof}

  Let $\H_1:=\{(x,y)\in\R^2: 0<y<d_1\}$,  $\H_2:=\{(x,y)\in\R^2: -d_2<y<0\}$. Let
 $\partial_+=\{(x,y)\in\R^2: y=d_1\}$, $\partial_0=\{(x,y)\in\R^2: y=0\}$, $\partial_-=\{(x,y)\in\R^2: y=-d_2\}$. Let
 $(A_1A_2A_3A_4A_5)$ be a strictly convex pentagon such that
 $$ A_1\in \partial_+, A_2\in \H_1\cup \partial_+, \overline{A_3A_5}\subset \partial_0, A_4\in\partial_- .$$
 \begin{lemma}\label{lem:twostrips}
   Let $(A_1A_2A_3A_4A_5)$, $\H_1,\H_2,\partial_+,\partial_0,\partial_-$ be as above. Let $\i:(A_1A_2A_3A_4A_5)\to(X,\omg;\Sigma)$ be an admissible map. Suppose that $\i(\overline{A_2A_3})\neq \i(\overline{A_5A_1})$.
   \begin{itemize}
     \item If $d_1\geq d_2$, then there exists $A_6\in\overline{\H}_1$ such that $\i$ can be admissibly extended to $(A_1A_2A_3A_4A_5A_6)$.  If in addition,   $\i(\inter(A_1A_2A_3A_5))$ is not  contained in  any horizontal cylinder, then $A_6\in \H_1$.
     \item If $d_1\leq d_2$, then there exist $A_7\in\overline{\H}_2$ such that $\i$ can be admissibly extended to $(A_1A_2A_3A_4A_7A_5)$, and $A_8\in\overline{\H}_2$ such that $\i$ can be admissibly extended to $(A_1A_2A_3A_8A_4A_5)$.
   \end{itemize}
 \end{lemma}
\begin{proof}
The proof is similar to that of Lemma \ref{lem:emdedding:strip}.
\end{proof}

  \begin{figure}[b]
  \centering
 \begin{tikzpicture}
  \path[fill,lightgray](4.25,-7.5)--(-5,-7.5)--
   (-5,-10.5)--(4.25,-10.5)--cycle;
   \path[fill,black]
   (-3+5,-4-0.5-4.5) coordinate(c1) circle(0.05)
                node[above]{{\tiny $A_5$}}
   (-6+5,-3-4.5) coordinate(c2) circle(0.05)
                node[above]{\tiny $A_1$}
   (-9+5,-4-0.5-4.5) coordinate(c3) circle(0.05)
                node[left]{\tiny $A_2$}
   (-2.5,-5-1-4.5) coordinate(c4) circle(0.05)
                node[below]{\tiny $A_3$}
   (-4.5+5,-5-1-4.5) coordinate(c5) circle(0.05)
                node[below]{\tiny $A_4$};
  \draw[dashed]
  (4.25,-7.5)--(-5,-7.5)
   (-5,-10.5)--(4.25,-10.5);
   \draw (c1)--(c2)--(c3)--(c4)--(c5)--cycle;
   \draw (c4)--(c2)--(c5);
   \draw (-5.3+5,-4-4.5) node{\tiny$\gm_3$}
         (-6.7+5,-4-4.5) node{\tiny$\gm_1$}
         (-5.7+5,-6-4.5) node[below]{\tiny $\gm_2$}
         (3.5,-9)node{\tiny$\mathcal{M}$}
        ;
 \end{tikzpicture}
  \caption{Pentagon-extension of triangles I: the triangle $(A_1A_3A_4)$ is not contained in any cylinder.}\label{Fig:convexq}
\end{figure}
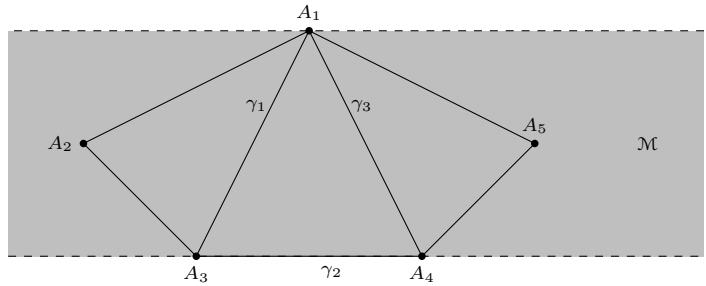

  \begin{corollary}\label{cor:pentagon:strictlyconvex}
   Let $\gm_1,\gm_2,\gm_3 $ be three saddle connections bounding a triangle $\Dt$ on $ (X,\omg;\Sigma)$ which is not contained in any  cylinder. Then there exist a convex pentagon $\mathcal P=(P_1P_2P_3P_4P_5)\subset\R^2$ which is strictly convex at each vertex, and an admissible map  $\i:\mathcal{P}\to (X,\omg;\Sigma)$ (see Figure \ref{Fig:convexq}), such that
  $$\i(\overline{P_1P_3})=\gm_1,~\i(\overline{P_1P_4})=\gm_2,~ \i(\overline{P_3P_4})=\gm_3.$$
\end{corollary}

\begin{proof}
   Suppose that $\gm_1,\gm_2,\gm_3$ are arranged in the counterclockwise order with respect to $\Dt$.
  After applying the $\GL(2,R)$ action, we may assume that $\gm_2$ is horizontal.  Let $(A_1A_3A_4)$ be a triangle in $\R^2$  and $\i: (A_1A_3A_4)\to (X,\omg;\Sigma)$  an admissible map such that $\i(\overline{A_1A_3})=\gm_1$, $\i(\overline{A_3A_4})=\gm_2$ and $\i(\overline{A_4A_1})=\gm_3$.

 Let $\mathcal{M}\subset\R^2$ be the horizontal strip which contains $A_1, ~A_3,~A_4$ in the boundary.  Since $\Delta$ is not contained in any cylinder, by Lemma \ref{lem:emdedding:strip}, there exists $A_2\in \mathbf{int}(\mathcal{M})$ on the left side of $\overline{A_1A_3}$
  such that $\i$ can be admissibly extended to $(A_1A_2A_3A_4)\to(X,\omg;\Sigma)$.  Consider the quadrilateral $(A_1A_2A_3A_4)$. It is not contained in any cylinder.  Again, it follows from Lemma \ref{lem:emdedding:strip} that there exists $A_5\in\mathbf{int}(\mathcal{M})$ on the right side of $\overline{A_1A_4}$ such that
  $\i$ can be admissibly extended to~　$\i:(A_1A_2A_3A_4A_5)\to(X,\omg;\Sigma)$. Since  $A_1, ~A_3,~A_4$ belong to the boundary of $\mathcal{M}$ while $A_2$ and $A_5$ belong to interior, we see that $(A_1A_2A_3A_4A_5)$ is strictly convex at each vertex.
\end{proof}

 \begin{lemma}\label{lem:pentagon:cylinder}
  Let $\gm_1,\gm_2,\gm_3 $ be three saddle connections bounding a triangle $\Dt$ on $ (X,\omg;\Sigma)$ which is  contained in a non-simple  cylinder $\mathcal{C}$. Suppose that  $\gm_3$ is contained in the boundary of $\C$. Then there exist a convex pentagon $(A_1A_2A_3A_4A_5)\subset\R^2$  and an admissible map  $\i:(A_1A_2A_3A_4A_5)\to (X,\omg;\Sigma)$, such that
  \begin{itemize}
    \item $\i(\overline{A_1A_3})=\gm_1,~\i(\overline{A_3A_4})=\gm_2,~ \i(\overline{A_1A_4})=\gm_3$;
    \item the interior angle  at $A_1$ is $\pi$ while the other interior angles are less than $\pi$ (see Figure \ref{Fig:pentagon}(a)), or the interior angle of at $A_4$ is $\pi$ while the other interior angles are less than $\pi$ (see Figure \ref{Fig:pentagon}(b)).
  \end{itemize}
\end{lemma}

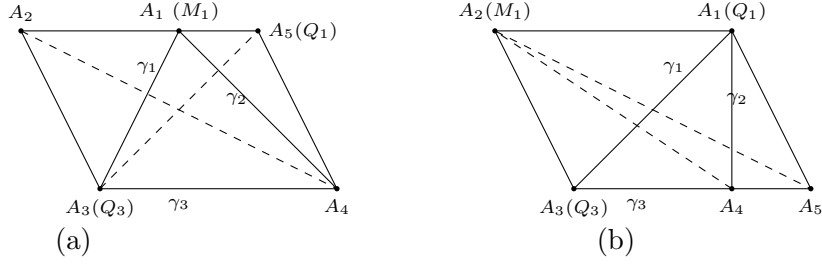
\begin{figure}[t]
  \begin{tikzpicture}[scale=0.7]


       \path[fill,black]
   (3+2-9,-3) coordinate(c2) circle(0.05) node[above]{\tiny $A_2$}
   (6+2-9,-3) coordinate(c1) circle(0.05)node[above]{\tiny $A_1$ ($M_1$)}
   (9+2-9,-6) coordinate(c4) circle(0.05)node[below]{\tiny $A_4$}
   (7.5+2-9,-3) coordinate(c5) circle(0.05)node[right]{\tiny $A_5$($Q_1$)}
   (4.5+2-9,-5-1) coordinate(c3) circle(0.05)node[below]{\tiny $A_3$($Q_3$)};
   \draw (c1)--(c2)--(c3)--(c4)--(c5)--cycle;
   \draw (c3)--(c1)--(c4);
   \draw[dashed](c2)--(c4) (c3)--(c5);
   \draw (5.4+2-9,-3.7) node{\tiny$\gm_1$}
         (7.1+2-9,-4.3) node{\tiny$\gm_2$}
         (6+2-9,-6) node[below]{\tiny $\gm_3$}
          (-3,-7)node{(a)};


       \path[fill,black]
   (3+2,-3) coordinate(c1) circle(0.05) node[above]{\tiny $A_2$($M_1$)}
   (7.5+2,-3) coordinate(c2) circle(0.05)node[above]{\tiny $A_1$($Q_1$)}
   (9+2,-6) coordinate(c3) circle(0.05)node[below]{\tiny $A_5$}
   (7.5+2,-5-1) coordinate(c4) circle(0.05)node[below]{\tiny $A_4$}
   (4.5+2,-5-1) coordinate(c5) circle(0.05)node[below]{\tiny $A_3$($Q_3$)};

   \draw (c1)--(c2)--(c3)--(c4)--(c5)--cycle;
   \draw (c4)--(c2)--(c5);
   \draw[dashed](c4)--(c1)--(c3);
   \draw (5.4+3,-3.7) node{\tiny$\gm_1$}
         (6.8+2.8,-4.3) node{\tiny$\gm_2$}
         (5.7+2,-6) node[below]{\tiny $\gm_3$}
          (-0.7+2+6,-7)node{(b)};
  \end{tikzpicture}
  \caption{Pentagon-extension of triangles II: the triangle $(A_1A_3A_4)$ is contained in some non-simple cylinder.}\label{Fig:pentagon}
\end{figure}

\begin{proof}
There exist a parallelogram $(Q_1Q_2Q_3Q_4)$ on $\R^2$, and an admissible map
 $ \mathcal{J}: (Q_1M_1\cdots M_n Q_2 Q_3 N_1\cdots N_k Q_4)\to (X,\omg;\Sigma)$ whose image is $\mathcal{C}$, where $M_1,\cdots,M_n \in \overline{Q_1Q_2}$ and $N_1,\cdots,N_k\in\overline{Q_3Q_4}$.   By assumption, $\mathcal{C}$ is not simple. It follows  that $\max\{n,k\}\geq1$.
 \begin{itemize}
    \item If  $n\geq1$, we may choose the polygon $(Q_1M_1\cdots M_n Q_2 Q_3 N_1\cdots N_k Q_4)$ such that
         $ \i(\overline{M_1Q_3})=\gm_1 $. Let
         $A_1=M_1$, $A_3=Q_3$,  $A_5=Q_1$, $A_2=M_2$ if $n>1$ or $A_2=Q_2$ if $n=1$, and $A_4=Q_4$ if $k=0$ or $A_4=N_1$ if $k\geq1$.
   \item If $k\geq1$, we may choose the polygon $(Q_1M_1\cdots M_n Q_2 Q_3 N_1\cdots N_k Q_4)$ such that $\i(\overline{Q_1Q_3})=\gm_1$. Let  $A_1=Q_1$, $A_3=Q_3$. $A_4=N_1$, $A_2=Q_2$ if $n=0$ or $A_2=M_1$ if $n\geq1$, and  $A_5=N_2$ if $k>1$ or $A_5=Q_4$ if $k=1$.
 \end{itemize}
 Then the resulting pentagon $(A_1A_2A_3A_4A_5)$ and the associate admissible map $\mathcal{J}:(A_1A_2A_3A_4A_5)\to(X,\omg;\Sigma)$ satisfy all the properties we want.
\end{proof}

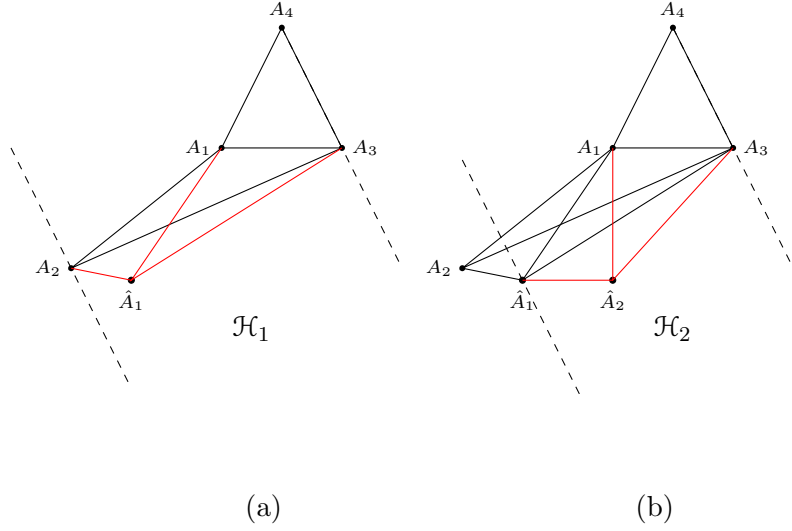
\begin{figure}
 \begin{tikzpicture}[scale=0.8]
      \path[fill,black]
    (1.5,0) coordinate (a1) circle (0.05)node[above]{\tiny$A_4$}
    (0.5,-2) coordinate (a4) circle (0.05)node[left]{\tiny$A_1$}
    (2.5,-2) coordinate (a2) circle (0.05)node[right]{\tiny$A_3$}
    (-2,-4) coordinate (a3) circle (0.05)node[left]{\tiny$A_2$};

    \draw[dashed](a1)--(3.5,-4)
    (-3,-2)--(-1,-6);
       \draw[fill,black] (-1,-4.2)node[below]{\tiny$\hat{A}_1$} circle (0.05);
    \draw(a1)--(a2)--(a3)--(a4)--cycle
         (a2)--(a4);
    \draw[ red](a4)--(-1,-4.2)--(a2) (a3)--(-1,-4.2) ;
    \draw(1.2,-8)node{(a)}
    (1,-5)node{$\H_1$};


   \path[fill,black]
    (1.5+6.5,0) coordinate (a1) circle (0.05)node[above]{\tiny$A_4$}
    (0.5+6.5,-2) coordinate (a4) circle (0.05)node[left]{\tiny$A_1$}
    (2.5+6.5,-2) coordinate (a2) circle (0.05)node[right]{\tiny$A_3$}
    (-2+6.5,-4) coordinate (a3) circle (0.05)node[left]{\tiny$A_2$};

    \draw[dashed](a1)--(3.5+6.5,-4)
    (-2+6.5,-2.2)--(+6.5,-6.2);
       \draw[fill,black] (-1+6.5,-4.2)node[below]{\tiny$\hat{A}_1$} circle (0.05);
       \draw[fill,black] (0.5+6.5,-4.2)node[below]{\tiny$\hat{A}_2$} circle (0.05);
    \draw(a1)--(a2)--(a3)--(a4)--cycle
         (a2)--(a4);
    \draw(a4)--(-1+6.5,-4.2)--(a2) (a3)--(-1+6.5,-4.2) ;
    \draw[red](-1+6.5,-4.2)--(0.5+6.5,-4.2)--(a2) (0.5+6.5,-4.2)--(a4);
    \draw(1.2+6.5,-8)node{(b)}
    (1.5+6.5,-5)node{$\H_2$};
 \end{tikzpicture}
  \caption{\small{Constructing a strictly convex admissible quadrilateral $(A_1\hat{A}_mA_3A_4)$ from an arbitrary admissible quadrilateral $(A_1A_2A_3A_4)$}.}\label{fig:coconvex}
\end{figure}

\begin{lemma}\label{lem:coconvex}
Let  $\i: (A_1A_2A_3A_4)\to (X,\omg;\Sigma)$ be  an admissible map, where  $(A_1A_2A_3A_4)\subset\R^2$  is   strictly convex at $A_2,A_3,A_4$. 
 If $(A_1A_2A_3A_4)$ is not strictly convex at $A_1$, then  there exists a sequence of  admissible maps $\hat{\i}_k:(A_1\hat{A}_{k-1}\hat{A}_k A_3A_4)\to (X,\omg;\Sigma)$, where $k=1,2,\cdots,m$ and $\hat{A}_0=A_2$,  such that
   \begin{enumerate}[(i)]
     \item $\i$ and $\hat{\i}_k$ coincide on the triangle $(A_1A_3A_4)$;
     \item $(A_1\hat{A}_{k-1}\hat{A}_kA_3)$ is strictly convex at each vertex for $k=1,2,\cdots,m$;
     \item $(A_1\hat{A}_m A_3A_4)$ is  strictly convex at each vertex.
   \end{enumerate}
\end{lemma}
\begin{proof}
   After apply an action of $\GL(2,\R)$, we may assume that $(A_1A_3A_4)$ is an equilateral triangle (see Figure \ref{fig:coconvex}(a)).
  We shall proceed by induction. The algorithm is as following.

 \begin{quotation}
       Consider the strip $\H_1$ directed by $\overrightarrow{A_4A_3}$  which contains $A_2$, $A_3$, $A_4$ in the boundary (see Figure \ref{fig:coconvex}(a)). Since $(A_1A_2A_3A_4)$ is not strictly convex at $A_1$, it follows that $A_1\in\mathbf{int}(\H_1)$ and that  $\i(\inter(A_1A_2A_3A_4))$ by $\i$ is not contained in any cylinder in the direction of $\overrightarrow{A_4A_3}$. By Lemma \ref{lem:emdedding:strip}, there exist $\hat{A}_1\in\mathbf{int}(\H_1)$ below  $\overline{A_2A_3}$ and an  admissible
       $\hat{\i}_1:(A_1A_2\hat{A}_1A_3A_4)\to (X,\omg;\Sigma) $  which coincides with $\i$ on $(A_1A_3A_4)$, such that $(A_1\hat{A}_0\hat{A}_1A_3)$ is strictly convex at each vertex, where $\hat{A}_0=A_2$.
     Let  $\hat{\theta}_{13}$ be the interior angle of $(A_1\hat{A}_0\hat{A}_1A_3)$ at $A_3$. Then
       $$ \hat{\theta}_{03} <\hat{\theta}_{13}<2\pi/3,$$
       where $\hat{\theta}_{03}$ is the interior angle of $(A_1A_2A_3)$ at $A_3$ and $2\pi/3$ is the exterior angle of $(A_1A_3A_4)$ at $A_3$.

     If $(A_1\hat{A}_1A_3A_4)$ is strictly convex at $\hat{A}_1$, the algorithm terminates. Otherwise, we repeat the construction above for $(A_1\hat{A}_1A_3A_4)$.
  \end{quotation}

  We now show that the algorithm will terminate after finitely many steps.  Suppose to the contrary that the algorithm will never terminate.
  In each step, we construct an admissible map $\hat{\i}_k:(A_1\hat{A}_{k-1}\hat{A}_k A_3A_4)\to (X,\omg;\Sigma)$, where $k=1,2,\cdots,m$ and $\hat{A}_0=A_2$,  such that
   \begin{enumerate}[(i)]
     \item $\i$ and $\hat{\i}_k$ coincide on the triangle $(A_1A_3A_4)$;
     \item $(A_1\hat{A}_{k-1}\hat{A}_kA_3)$ is strictly convex at each vertex for $k=1,2,\cdots,m$;
     \item $(A_1\hat{A}_m A_3A_4)$ is not strictly convex at $A_1$.
   \end{enumerate}
  Let  $\hat{\theta}_{k3}$ be the interior angle of $(A_1\hat{A}_{k-1}\hat{A}_kA_3)$ at $A_3$. Then  $\hat{\theta}_{03}<\hat{\theta}_{k3}<\hat{\theta}_{k+1,3}<2\pi/3$ for all $k\geq1$. In particular, $\i(\overline{A_3\hat{A}_k})\neq \i(\overline{A_3\hat{A}_j}) $  for all $k>j>1$.
     Since  $\i:(A_1\hat{A}_{k-1}\hat{A}_kA_3)\to (X,\omg;\Sigma) $ is an admissible map, we have
     $$ \mathbf{Area}(A_1\hat{A}_k A_3)=\frac{1}{2}|\overline{A_1A_3}|\cdot |\overline{A_3\hat{A}_k}|\sin \hat{\theta}_{k3}<\mathbf{Area}(X,\omg), $$
     which implies that
     $$|\overline{A_3\hat{A}_k}|<{2\mathbf{Area}(X,\omg)}
     |\overline{A_1A_3}|^{-1}(\min\{|\sin \hat{\theta}_{03}|,\sqrt{3}/2\})^{-1}.$$
     Let $\mathbf{T}:={2\mathbf{Area}(X,\omg)}
     |\overline{A_1A_4}|^{-1}(\min\{|\sin\hat{\theta}_{03}|,\sqrt{3}/2\})^{-1}.$
     On the other hand, there are only finitely many saddle connections on $(X,\omg;\Sigma)$ whose length are less than $\mathbf{T}$. This is a contradiction which proves the lemma.

\end{proof}

\section{Simple cylinder preserving}
 In the remaining of this paper, we assume that  $(X,\omg;\Sigma)$ and $(X',\omg';\Sigma')$ are  half-translation surfaces with marked points, and that $$F:\ms(X,\omg;\Sigma)\to\ms(X',\omg';\Sigma')$$
 is an isomorphism.
The goal of this section is to prove that $F$ preserves simple cylinders (see Proposition \ref{lem:simplecylinder}).
Let us start  with the following two lemmas, which will be frequently used in the sequel.

\begin{lemma}[Triangle lemma]\label{lem:triangle}
  Let $\i:(A_1A_2A_3)\to(X,\omg;\Sigma)$ be an admissible map, where $(A_1A_2A_3)$ is a triangle. Then $\i(\overline{A_1A_2})$, $\i(\overline{A_2A_3})$ and $\i(\overline{A_3A_1})$ are pairwise different.
\end{lemma}
\begin{proof}
  Suppose to the contrary that $\i(\overline{A_{i-1}A_i})=\i(\overline{A_iA_{i+1}})$ for some $i\in\{1,2,3\}$. Since $(X,\omg;\Sigma)$ is a half-translation surface, it then follows that the interior angle of $(A_1A_2A_3)$ at $A_i$ is $\pi$. This is a contradiction which proves the lemma.
\end{proof}

\begin{lemma}[Quadrilateral lemma ]\label{lem:quadri:diagonals}
   Let $\i:(A_1A_2A_3A_4)\to (X,\omg;\Sigma)$ be an admissible map, where $(A_1A_2A_3A_4)$  is a strictly convex quadrilateral. Then there exist  a  strictly convex quadrilateral $ (A'_1A_2'A_3'A_4')$ , and  an admissible map $$\i': (A'_1A_2'A_3'A_4')\to(X',\omg';\Sigma')$$ such that
   \begin{enumerate}[(i)]
     \item  $F\circ\i(\overline{A_1A_3})=\i'(\overline{A_1'A_3'})$ and
     $F\circ\i(\overline{A_2A_4})=\i'(\overline{A_2'A_4'})$;
     \item for any triangulation $\Gm$ of $(X,\omg;\Sigma)$ which contains $\i(\overline{A_1A_2})$, $\i(\overline{A_2A_3})$, $\i(\overline{A_3A_4})$, $\i(\overline{A_4A_1})$, $ F(\Gm)$ is a triangulation of $(X',\omg';\Sigma')$ which contains
     $ \i'(\overline{A'_1A'_2})$, $\i'(\overline{A'_2A'_3})$, $\i'(\overline{A'_3A'_4})$,
     $\i(\overline{A'_4A'_1})$.
   \end{enumerate}
   If in addition, $\i$ can be admissibly extended to a pentagon $(A_1A_2A_3A_4A_5)$ which contains a diagonal $\overline{A_5A_2}$, then $F\circ\i(\overline{A_1A_4})=\i'(\overline{A_i'A_{i+1}'})$ for some $i=1,2,3,4$, where $A_5'=A_1'$.
\end{lemma}

\begin{proof}
  By Lemma \ref{lem:triangulation:correspondence}, $F(\Gm)$ is a triangulation of $(X',\omg';\Sigma')$ which contains $F\circ\i(\overline{A_1A_3})$. Therefore, there exists an admissible map $\i':(A_1'A_2'A_3'A_4')\to(X',\omg';\Sigma')$ such that \begin{itemize}
    \item $F\circ\i(\overline{A_1A_3})=\J(\overline{A_1'A_3'})$ and
     \item $\i'(\overline{A_{i-1}'A_{i}'})\in F(\Gm)$ for each  $i=1,2,3,4$, where $A'_{0}=A'_4$.
  \end{itemize}

  Notice that $\i(\overline{A_2A_4})$  intersects $\i(\overline{A_1A_3})$ and intersects no other saddle connections from $\Gm$. Therefore,
  $F\circ\i(\overline{A_2A_4})$ intersects $\i'(\overline{A'_1A'_3})$  and intersects no other saddle connections from $F(\Gm)$. Then $(A_1'A_2'A_3'A_4')$ is strictly convex, and
   $F\circ\i(\overline{A_2A_4})=\J(\overline{A_2'A_4'})$.

   If in addition, $\i$ can be admissibly extended to a pentagon $(A_1A_2A_3A_4A_5)$ which contains a diagonal $\overline{A_5A_2}$, let $\Gm_1$ be a triangulation of $(X,\omg;\Sigma)$ which contains $\i(\overline{A_1A_4})$, $\i(\overline{A_1A_3})$, and all images of sides of $(A_1A_2A_3A_4A_5)$ by $\i$. Then
   $\i(\overline{A_5A_2})$  intersects both $\i(\overline{A_1A_4})$ and $\i(\overline{A_1A_3})$ but intersects no other saddle connections from $\Gm_1$. Correspondingly,  $F\circ\i(\overline{A_5A_2})$ intersects both $F\circ\i(\overline{A_1A_4})$ and $\i'(\overline{A_1'A_3'})$  but intersects no other saddle connections from $F(\Gm_1)$. Consequently,  $F\circ\i(\overline{A_1A_4})=\i'(\overline{A_i'A_{i+1}'})$ for some $i=1,2,3,4$, where $A_5'=A_1'$.

\end{proof}

 \begin{proposition}[Simple cylinder]\label{lem:simplecylinder}
    Let $\i:(Q_1Q_2Q_3Q_4)\to(X,\omg;\Sigma)$ be an admissible map such that $\i(Q_1Q_2Q_3Q_4)$ is a simple cylinder which contains $\i(\overline{Q_1Q_2})$ and $\i(\overline{Q_3Q_4})$ as boundary components. Then there exists  an   admissible map $\i':(Q'_1Q'_2Q'_3Q'_4)\to(X',\omg';\Sigma')$  such that
    \begin{itemize}
      \item $F\circ\i(\overline{Q_iQ_j})=\i'(\overline{Q_i'Q_j'})$ for all diagonals and sides
      \item
      $\i'(Q'_1Q'_2Q'_3Q'_4)$ is a simple cylinder which contains $\i'(\overline{Q_1'Q_2'})$ and $\i'(\overline{Q'_3Q'_4})$ as boundary components.
    \end{itemize}

 \end{proposition}

 \begin{figure}[b]
  \begin{tikzpicture}[scale=0.7]


       \path[fill,black]
   (3+2-9,-3) coordinate(c2) circle(0.05) node[above]{\tiny $Q_2$}
   (6+2-9,-2) coordinate(c1) circle(0.05)node[above]{\tiny $Q_5$}
   (9+2-9,-6) coordinate(c4) circle(0.05)node[below]{\tiny $Q_4$}
   (7.5+2-9,-3) coordinate(c5) circle(0.05)node[above]{\tiny $Q_1$}
   (4.5+2-9,-5-1) coordinate(c3) circle(0.05)node[below]{\tiny $Q_3$};
   \draw (c1)--(c2)--(c3)--(c4)--(c5)--cycle;
   \draw(c2)--(c4) (c3)--(c5) (c2)--(c5);
   \draw
          (-3+2,-7)node{(a)};


           \path[fill,black]
   (3+2-9+8,-3) coordinate(c2) circle(0.05) node[left]{\tiny $Q'_2$}
   (7.5+2-9+8,-3) coordinate(c1) circle(0.05)node[right]{\tiny $Q'_1$}
   (9+2-9+8,-6) coordinate(c4) circle(0.05)node[below]{\tiny $Q'_4$}
   (7.5+2-9+5.5,-2) coordinate(c5) circle(0.05)node[above]{\tiny $Q'_5$}
   (4.5+2-9+8,-5-1) coordinate(c3) circle(0.05)node[below]{\tiny $Q'_3$}
   (2.2+2-9+8,-5)circle(0.05)
   (3+2-9+8,-4.5) node{\tiny $\Delta'$};
   \draw (c1)--(c2)--(c3)--(c4)--cycle;
   \draw (c3)--(c1)--(c4) (c1)--(c5)--(c2) ;
   \draw(c2)--(c4) ;
   \draw[dashed] (c2)--(2.2+2-9+8,-5)--(c3);

   \draw
          (-3+8+2,-7)node{(b)};
  \end{tikzpicture}
  \caption{Correspondence between Simple cylinders. }\label{Fig:simplecylinder:preserving}
\end{figure}
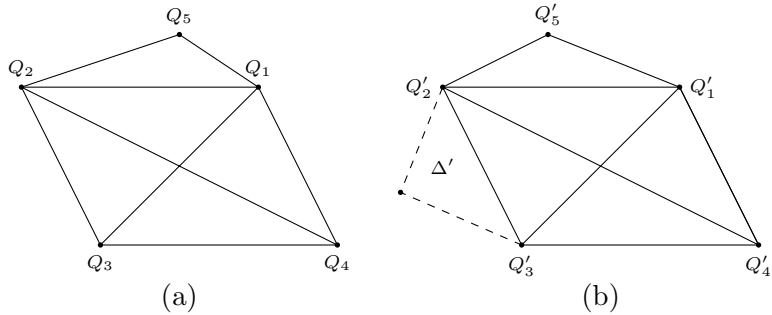
 \begin{proof}
    Let $\Gm$ be
     triangulation of $(X,\omg;\Sigma)$ which contains $\i(\overline{Q_1Q_2})$, $\i(\overline{Q_2Q_3})=\i(\overline{Q_4Q_1})$, $\i(\overline{Q_3Q_4})$, $\i(\overline{Q_3Q_1})$. By Lemma \ref{lem:quadri:diagonals}, there exist  a  strictly convex quadrilateral $ (Q'_1Q_2'Q_3'Q_4')$, and  an admissible map $$\i': (Q'_1Q_2'Q_3'Q_4')\to(X',\omg';\Sigma')$$ such that
   \begin{itemize}
     \item  $F\circ\i(\overline{Q_1Q_3})=\i'(\overline{Q_1'Q_3'})$ and
     $F\circ\i(\overline{Q_2Q_4})=\i'(\overline{Q_2'Q_4'})$;
     \item  $ F(\Gm)$ is a triangulation of $(X',\omg';\Sigma')$ which contains
     $ \i'(\overline{Q'_1Q'_2})$, $\i'(\overline{Q'_2Q'_3})$, $\i'(\overline{Q'_3Q'_4})$,
     $\i(\overline{Q'_4Q'_1})$.
   \end{itemize}

    If $(X,\omg;\Sigma)$ is a torus with one marked point, then $\Gm_1:=\{\i(\overline{Q_1Q_2}),\i(\overline{Q_2Q_3}),$ $\i(\overline{Q_1Q_3}\}$ is a triangulation. Therefore,   $F(\Gm_1):=\{F\circ\i(\overline{Q_1Q_2}),F\circ\i(\overline{Q_2Q_3}),
    F\circ\i(\overline{Q_1Q_3})\}$ is a triangulation of $(X',\omg';\Sigma')$. Then $\i'(\overline{Q_1'Q_2'})=\i'(\overline{Q_3'Q_4'})$ and $\i'(\overline{Q_2'Q_3'})=\i'(\overline{Q_4'Q_1'})$. The proposition follows.

    In the following, we assume that $(X,\omg;\Sigma)$ is not a torus with one marked point. Let $\eta_1,\eta_2,\eta_3$ be three different interior saddle connections of $\C$ such that each of them intersects both $\i(\overline{Q_1Q_3})$ and $\i(\overline{Q_2Q_3})$. Then each of $\{F(\eta_1),F(\eta_2),F(\eta_3)\}$  intersects both   $\i'(\overline{Q'_1Q'_3})$ and $F\circ\i(\overline{Q_2Q_3})$ while disjoint with other saddle connections from $F(\Gm)$. Therefore, $F\circ\i(\overline{Q_2Q_3})=\i'(\overline{Q_{i-1}'Q_{i}'})$ for some $i=1,2,3,4$, where $Q_0'=Q_4'$. Up to relabelling, we may suppose that  $F\circ\i(\overline{Q_2Q_3})=\i'(\overline{Q_2'Q_3'})$.
    By Lemma \ref{lem:triangle}, we see that
  $$\i'(\overline{Q_2'Q_3'})\neq \i'(\overline{Q_1'Q_2'}),~\i'(\overline{Q_2'Q_3'})\neq \i'(\overline{Q_3'Q_4'}).$$
   If $\i'(\overline{Q_2'Q_3'})\neq\i'(\overline{Q_1'Q_4'})$, consider the triangles  from the complement of $F(\Gm)$  on  $(X',\omg';\Sigma')$. Let $\Delta'$ be the one which contains $\i'(\overline{Q_2'Q_3'})$ as a boundary and which is not contained in the image of $(Q_1'Q_2'Q_3'Q_4')$ under the admissible map $\i'$ (see Figure \ref{Fig:simplecylinder:preserving}(b) ). The other two boundary saddle connections of $\Delta'$ can not be $\i'(\overline{Q_2'Q_3'})$. In this case, there are at most two saddle connections of $(X',\omg';\Sigma')$ which intersects both  $\i'(\overline{Q_2'Q_3'})$ and $\i'(\overline{Q_1'Q_3'})$ while disjoint with other saddle connections from $F(\Gm)$, where the maximum holds only if $\Delta'\cup(Q_1'Q_2'Q_3'Q_4')$ is a pentagon strictly convex at each vertex. Therefore,  $\i'(\overline{Q_1'Q_4'})=\i'(\overline{Q_2'Q_3'})$. In particular, the image of $(Q_1'Q_2'Q_3'Q_4')$ by $\i'$ is a simple cylinder on $(X',\omg';\Sigma')$.

   Next, we consider  $F\circ \i(\overline{Q_1Q_2})$ and $F\circ\i(\overline{Q_3Q_4})$. Since $(X,\omg;\Sigma)$ is not a torus with one marked point, it follows that  $\i(\overline{Q_1Q_2})\neq \i(\overline{Q_3Q_4})$. By Lemma \ref{lem:emdedding:strip},  we can admissibly extend $\i$ to  a convex pentagon $(Q_1Q_5Q_2Q_3Q_4)$ which is strictly convex at $Q_5$ (see Figure \ref{Fig:simplecylinder:preserving}(a)). By Lemma \ref{lem:quadri:diagonals}, this implies that $F(\overline{Q_1Q_2})$ is $\i'(\overline{Q_1'Q_2'})$ or $\i'(\overline{Q_3'Q_4'})$, say $\i'(\overline{Q_1'Q_2'})$. Similarly, it can be shown that $F\circ\i(\overline{Q_3Q_4})=\i'(\overline{Q_3'Q_4'})$.
 \end{proof}

\section{Triangle preserving}\label{sec:triangle}

 The goal of this section is to prove the following theorem.

\begin{theorem}\label{thm:triangle}
  Let $\gm_1,\gm_2,\gm_3$ be three saddle connections on $(X,\omg;\Sigma)$ which  bound a triangle.  Then $F(\gm_1),F(\gm_2),F(\gm_3)$ also bound a triangle on $(X',\omg';\Sigma')$.
\end{theorem}

Our strategy is to show that $F$ preserves convex pentagons.
 \begin{theorem}[Pentagon preserving]\label{thm:pentagon:preserving}
     Let $\i:(A_1A_2A_3A_4 A_5)\to (X,\omg;\Sigma)$ be an admissible map, where $(A_1A_2A_3A_4 A_5)$  is a convex pentagon strictly convex at $A_2,A_3,A_4,A_5$.  Then there exists an admissible map $$\i': (A'_1A_2'A_3'A_4'A'_5)\to(X',\omg';\Sigma'),$$ where $ (A'_1A_2'A_3'A_4'A'_5)$ is strictly convex at  $A_2',A_3',A_4',A_5'$, such that
     $$ F\circ \i(\overline{A_iA_j})=\i'(\overline{A_i'A_j'}) $$
     for all sides and diagonals.
 \end{theorem}

 \begin{proof}[Proof of Theorem \ref{thm:triangle} assuming Theorem \ref{thm:pentagon:preserving}]
  Let $\Delta$ be a triangle bounded by $\gm_1,\gm_2,\gm_3$. If $\Delta$ is contained in some simple cylinder, the theorem follows from Proposition \ref{lem:simplecylinder}. If $\Delta$ is not contained in any simple cylinder, it  then follows from Lemma \ref{lem:pentagon:cylinder} and Corollary \ref{cor:pentagon:strictlyconvex} that there exists an admissible
  map $\i:(A_1A_2A_3A_4 A_5)\to(X,\omg;\Sigma)$, where
  $(A_1A_2A_3A_4 A_5)$  is a convex pentagon strictly convex at $A_2,A_3,A_4, A_5$, such that $\Delta$ is contained in the image of $(A_1A_2A_3A_4 A_5)$ by $\i$.  By Theorem \ref{thm:pentagon:preserving}, there exists an  admissible map $\i':(A'_1A_2'A_3'A_4'A'_5)\to (X',\omg';\Sigma')$, where $(A'_1A_2'A_3'A_4'A'_5)$ is a   convex pentagon, strictly convex at $A_2',A_3',A_4', A_5'$,  such that $F\circ\i(\overline{A_iA_j})=\i'(\overline{A_i'A_j'})$ for all diagonals and sides $\overline{A_iA_j}$. Since $\gm_1,\gm_2,\gm_3$ bound a triangle which is contained in the image  $(A_1A_2A_3A_4A_5)$ by $\i$, then $F(\gm_1),F(\gm_2),F(\gm_3)$ also bound a triangle which is contained in the image of  $(A'_1A_2'A_3'A_4'A'_5)$ by $\i'$. In particular, $F(\gm_1),F(\gm_2),F(\gm_3)$ bound a triangle on $(X',\omg';\Sigma')$.
  \end{proof}

  The reminder of this section is to prove Theorem \ref{thm:pentagon:preserving}.

\begin{figure}[t]
  \begin{tikzpicture}[scale=0.7]


       \path[fill,black]
   (3+2-9,-3) coordinate(c2) circle(0.05) node[above]{\tiny $A_2$}
   (6+2-9,-3) coordinate(c1) circle(0.05)node[above]{\tiny $A_1$}
   (9+2-9,-6) coordinate(c4) circle(0.05)node[below]{\tiny $A_4$}
   (7.5+2-9,-3) coordinate(c5) circle(0.05)node[above]{\tiny $A_5$}
   (4.5+2-9,-5-1) coordinate(c3) circle(0.05)node[below]{\tiny $A_3$};
   \draw[dashed] (c1)--(c2)--(c3)--(c4)--(c5)--cycle;
   \draw (c3)--(c1)--(c4);
   \draw(c2)--(c4) (c3)--(c5);
   \draw
          (-3+2,-7)node{(a)};


           \path[fill,black]
   (3+2-9+8,-2.7) coordinate(c2) circle(0.05) node[left]{\tiny $B'_2$($A_4'$)}
   (6+2-9+8,-3) coordinate(c1) circle(0.05)node[right]{\tiny $B'_1$($A_1'$)}
   (9+2-9+8,-6) coordinate(c4) circle(0.05)node[below]{\tiny $B'_4$($A_2'$)}
   (7.5+2-9+5.5,-2) coordinate(c5) circle(0.05)node[above]{\tiny $B'_5$($A_5'$)}
   (4.5+2-9+8,-5-1) coordinate(c3) circle(0.05)node[below]{\tiny $B'_3$($A_3'$)};
   \draw[dashed] (c1)--(c4)--(c3)--(c2)--(c5)--cycle;
   \draw (c3)--(c1)--(c2) ;
   \draw(c2)--(c4) (c3)--(c5);
   \draw
          (-3+8+2,-7)node{(b)};


       \path[fill,black]
   (3+2-9,-3-6) coordinate(c2) circle(0.05) node[above]{\tiny $A_2$}
   (6+2-9,-3-6+0.7) coordinate(c1) circle(0.05)node[above]{\tiny $A_1$}
   (9+2-9,-6-6) coordinate(c4) circle(0.05)node[below]{\tiny $A_4$}
   (7.5+2-9,-3-6) coordinate(c5) circle(0.05)node[above]{\tiny $A_5$}
   (4.5+2-9,-5-1-6) coordinate(c3) circle(0.05)node[below]{\tiny $A_3$};
   \draw[dashed] (c1)--(c2)--(c3)--(c4)--(c5)--cycle;
   \draw (c3)--(c1)--(c4);
   \draw(c2)--(c4) (c3)--(c5);
   \draw[red](c2)--(c5);
   \draw
          (-3+2,-7-6)node{(c)};


           \path[fill,black]
   (3+2-9+8,-2.7-6) coordinate(c2) circle(0.05) node[left]{\tiny $B'_2$($A_4'$)}
   (6+2-9+9.5,-3-6) coordinate(c1) circle(0.05)node[right]{\tiny $B'_1$($A_1'$)}
   (9+2-9+8,-6-6) coordinate(c4) circle(0.05)node[below]{\tiny $B'_4$($A_2'$)}
   (7.5+2-9+5.5,-2-6) coordinate(c5) circle(0.05)node[above]{\tiny $B'_5$($A_5'$)}
   (4.5+2-9+8,-5-1-6) coordinate(c3) circle(0.05)node[below]{\tiny $B'_3$($A_3'$)};
   \draw[dashed] (c1)--(c4)--(c3)--(c2)--(c5)--cycle;
   \draw (c3)--(c1) (c1)--(c2) ;
   \draw(c2)--(c4) (c3)--(c5);
   \draw[red](c5)--(c4);
   \draw
          (-3+8+2,-7-6)node{(d)};

  \end{tikzpicture}
  \caption{Correspondence between pentagons. In (a) and (b), the pentagons are not strictly convex at $A_1$ and $A_1'$. In (c) and (d), the pentagons are strictly convex at each vertex. }\label{Fig:pentagon:preserving}
\end{figure}
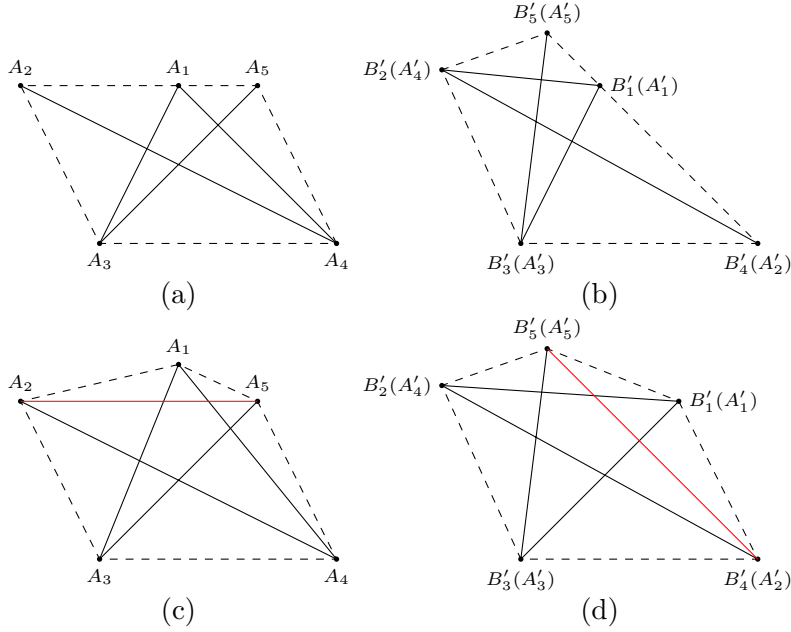
\begin{lemma}[Pentagon lemma I]\label{lem:pentagon:diagonals}
   Let $\i:(A_1A_2A_3A_4 A_5)\to (X,\omg;\Sigma)$ be an admissible map, where $(A_1A_2A_3A_4 A_5)$  is a convex pentagon strictly convex at $A_2,A_3,A_4,A_5$.  Then there exists an admissible map $$\i': (A'_1A_2'A_3'A_4'A'_5)\to(X',\omg';\Sigma'),$$ where $ (A'_1A_2'A_3'A_4'A'_5)$  is strictly convex at  $A_2',A_3',A_4',A_5'$ satisfy the following.
   \begin{enumerate}
      \item[\textup{(i)}]  If $(A_1A_2A_3A_4A_5)$ is strictly convex at each vertex, so is $(A_1'A_2'A_3'A_4'A_5')$.
     \item[\textup{(ii)}]  $F\circ\i(\overline{A_iA_j})=\i'(\overline{A_i'A_j'})$  for all diagonals $\overline{A_iA_j}$.
     \item[\textup{(iii)}] For any triangulation $\Gm$ of $(X,\omg;\Sigma)$ which contains $\i(\overline{A_1A_2})$, $\i(\overline{A_2A_3})$, $\i(\overline{A_3A_4})$, $\i(\overline{A_4A_5})$ and $\i(\overline{A_5A_1})$, $ F(\Gm)$ is a triangulation of $(X',\omg';\Sigma')$ which contains
     $ \i'(\overline{A'_1A'_2})$, $\i'(\overline{A'_2A'_3})$, $\i'(\overline{A'_3A'_4})$,
     $\i(\overline{A'_4A'_5})$, $\i'(\overline{A'_5A'_1}).$
   \end{enumerate}
\end{lemma}

\begin{proof}

   Let $\Gm$ be a triangulation of $(X,\omg;\Sigma)$ which contains $\i(\overline{A_2A_4})$, $\i(\overline{A_1A_4})$ and $\i(\overline{A_{i-1}A_{i}})$ for all $i=1,2,3,4,5$, where $A_{0}=A_5$. Then by Lemma \ref{lem:triangulation:correspondence},  $F(\Gm)$ is a triangulation of $(X',\omg';\Sigma')$ which contains $F\circ\i(\overline{A_2A_4})$,  $F\circ\i(\overline{A_1A_4})$  and $F\circ\i(\overline{A_{i-1}A_{i}})$ for all $i=1,2,3,4,5$, where $A_{0}=A_5$.

  Consider the admissible quadrilateral $(A_1A_2A_3A_4)$. By Lemma
   \ref{lem:quadri:diagonals}, there exists a strictly convex quadrilateral $(B_1'B_2'B_3'B_4')$, and an admissible map $\J:(B_1'B_2'B_3'B_4')\to(X',\omg';\Sigma')$ (see Figure \ref{Fig:pentagon:preserving}), such that
  \begin{itemize}
    \item $F\circ\i(\overline{A_1A_3})=\J(\overline{B_1'B_3'})$ and $F\circ\i(\overline{A_2A_4})=\J(\overline{B_2'B_4'})$;
    \item $F\circ\i(\overline{A_1A_4})=\i'(\overline{B_{i-1}'B_{i}'})$ for some   $i\in\{1,2,3,4\}$,  say $\overline{B_1'B_2'}$;
     \item $\i'(\overline{B_{i-1}'B_{i}'})\in F(\Gm)$ for each  $i=1,2,3,4$, where $B_{0}'=B_4'$;
  \end{itemize}


  Next, we consider the admissible pentagon $(A_1A_2A_3A_4A_5)$.
  Let $\Delta'$ be the triangle in $(X',\omg';\Sigma')\backslash F(\Gm)$  which contains $\i'(\overline{B_1'B_2'})$ in the boundary and  which is not contained in $\i'(B_1'B_2'B_3'B_4')$. Then we can admissibly extend $\i'$ to the pentagon $(B_1'B_4'B_3'B_2'B_5')$ such that \begin{itemize}
    \item $\i'(B_1'B_2'B_5')=\Delta'$;
    \item $F\circ\i(\overline{A_3A_5})=\J(\overline{B_3'B_5'})$, $F\circ\i(\overline{A_1A_4})=\J(\overline{B_1'B_2'})$;
    \item $\J(\overline{B_5'B_2'}),\J(\overline{B_5'B_1'})\in F(\Gm)$,
  \end{itemize}
  where for the second claim we use the fact that $\i(\overline{A_3A_5})$ intersects both $\i(\overline{A_2A_4})$ and $\i(\overline{A_1A_4})$ while disjoint from other saddle connections in $F(\Gamma)$, for the third claim we apply Lemma \ref{lem:quadri:diagonals} for the convex quadrilateral $(A_1A_3A_4A_5)$.

  In particular, $(B_1'B_4'B_3'B_2'B_5')$ is strictly convex at $B_5',B_2',B_3',B_4'$.
  By relabeling $A_1'=B_1',A_2'=B_4',A_3'=B_3',A_4'=B_2',A_5'=B_5'$ (see Figure \ref{Fig:pentagon:preserving}(b)(d)), we get the desired admissible map $\J:(A_1'A_2'A_3'A_4'A_5')\to (X',\omg';\Sigma')$.

  If in addition, $(A_1A_2A_3A_4A_5)$ is also strictly convex at $A_1$, then $F\circ\i(\overline{A_2A_5})$ intersects only $\J(\overline{A_1'A_3'})$ and $\J(\overline{A_1'A_4'})$ and disjoint with any other saddle connections in  $\J(\inter(A_1'A_2'A_3'A_4'A_5'))$. Therefore,  $(A_1'A_2'A_3'A_4'A_5')$ is also strictly convex at $A_1'$ and  $F\circ\i(\overline{A_2A_5})=\J(\overline{A_2'A_5'})$.
\end{proof}



 Next, we deal with the sides of $(A_1A_2A_3A_4A_5)$.
  For each diagonal  $\overline{A_{i-1}A_{i+1}}$, let
 \begin{equation*}
   \begin{array}{rl}
   &\mathcal{D}(\overline{A_{i-1}A_{i+1}})\\
   :=&\max \{\d(A_i,\overline{A_{i-1}A_{i+1}}),\d(A_{i-2},\overline{A_{i-1}A_{i+1}}),
        \d(A_{i+2},\overline{A_{i-1}A_{i+1}})\}
   \end{array}
 \end{equation*}
where $\d(\cdot,\cdot)$ represents the Euclidean distance on $\R^2$, and where $A_{-1}=A_4$, $A_0=A_5$, $A_6=A_1$, and $A_7=A_2$.
\begin{lemma}[Pentagon lemma II]\label{lem:pentagon:side1}
     Let $\i:(A_1A_2A_3A_4 A_5)\to (X,\omg;\Sigma)$  and    $\i':(A'_1A_2'A_3'A_4'A_5')\to (X',\omg';\Sigma')$ be as in  Lemma \ref{lem:pentagon:diagonals}. Suppose further that $(A_1A_2A_3A_4 A_5)$ is also strictly convex at $A_1$.
  \begin{enumerate}
    \item[\textup{(i)}] If $\d(A_3,\overline{A_{2}A_{4}})=\mathcal{D}(\overline{A_{2}A_{4}})$, then $$F(\{\i(\overline{A_2A_{3}}),\i(\overline{A_3A_{4}})\})=
        \{\i'(\overline{A_2'A_{3}'}),\i'(\overline{A_3'A_{4}'})\}.$$
     \item[\textup{(ii)}]If  $\d(A_{5},\overline{A_{2}A_{4}})=
        \d(A_{1},\overline{A_{2}A_{4}})=\mathcal{D}(\overline{A_{2}A_{4}})>
        \d(A_3,\overline{A_{2}A_{4}})$ and \\
        $\i(A_1A_2A_4A_5)$ is  contained in a cylinder $\C$ whose boundary contains $\i(\overline{A_2A_4})$,then
        $F\circ\i(\overline{A_{1}A_{2}})=\i'(\overline{A_{1}'A_{2}'}),$
        $F\circ\i(\overline{A_{4}A_{5}})=\i'(\overline{A_{4}'A_{5}'})$.
    \item[\textup{(iii)}]  If  $\d(A_{1},\overline{A_{2}A_{4}})=\mathcal{D}(\overline{A_{2}A_{4}})$ and $\i(A_1A_2A_4A_5)$ is not contained in any cylinder whose boundary contains $\i(\overline{A_2A_4})$, then \\ $F\circ\i(\overline{A_{1}A_{2}})=\i'(\overline{A_{1}'A_{2}'}).$
    \item[\textup{(iv)}]If  $\d(A_{5},\overline{A_{2}A_{4}})=\mathcal{D}(\overline{A_{2}A_{4}})$ and $\i(A_1A_2A_4A_5)$ is not contained in any cylinder whose boundary contains $\i(\overline{A_2A_4})$, then \\ $F\circ\i(\overline{A_{4}A_{5}})=\i'(\overline{A_{4}'A_{5}'}).$
  \end{enumerate}
  Similar results also hold for $\overline{A_3A_5}$, $\overline{A_1A_3}$, $\overline{A_1A_4}$, and $\overline{A_2A_5}$.
\end{lemma}

\begin{proof}
{The overall strategy is to admissibly extend $\i$ and $\i'$ to hexagons by using Lemma \ref{lem:extension1} and Lemma \ref{lem:twostrips}.}
  Since $(A_1A_2A_3A_4A_5)$ is strictly convex at each vertex,
 by relabeling  the vertices of $(A_1A_2A_3A_4A_5)$, we see that the cases of $\overline{A_3A_5}$, $\overline{A_1A_3}$, $\overline{A_1A_4}$, and $\overline{A_2A_5}$ are all equivalent to the case of $\overline{A_2A_4}$. In the following, we prove the case of $\overline{A_2A_4}$.

 (i). By Lemma \ref{lem:twostrips}, we see that $\i$ can be admissibly extend a hexagon $(A_1A_2A_6A_3A_4A_5)$ such that $(A_2A_6A_3A_4)$ is a strictly convex quadrilateral (see Figure \ref{Fig:pentagon:side1:234}(a)). Consider $(A_1A_2A_3A_4)$.
 It  follows from Lemma \ref{lem:quadri:diagonals} that
 \begin{equation}\label{eq:23:1234}
   F\circ\i(\overline{A_2A_3})\in\{\i'(\overline{A_1'A_2'}),\i'(\overline{A_2'A_3'}),
   \i'(\overline{A_3'A_4'}),\i'(\overline{A_4'A_1'})\}.
 \end{equation}
 Notice that $\i(\overline{A_6A_4})$ intersects $\i(\overline{A_2A_3})$ while disjoint from $\i(\overline{A_1A_4})$ and $\i(\overline{A_2A_5})$. Correspondingly, $F\circ \i(\overline{A_6A_4})$ intersects $F\circ\i(\overline{A_2A_3})$ while disjoint from $\i(\overline{A_1'A_4'})$ and $\i(\overline{A_2'A_5'})$. Therefore, (\ref{eq:23:1234}) is reduced to
\begin{equation}\label{eq:23:234}
   F\circ\i(\overline{A_2A_3})\in\{\i'(\overline{A_2'A_3'}),
   \i'(\overline{A_3'A_4'})\}.
 \end{equation}
 Similarly, we have
 \begin{equation}\label{eq:34:1234}
   F\circ\i(\overline{A_3A_4})\in\{\i'(\overline{A_2'A_3'}),
   \i'(\overline{A_3'A_4'})\}.
 \end{equation}
  Since $\i(\overline{A_2A_3})\neq\i(\overline{A_3A_4})$ by Lemma \ref{lem:triangle}, it then follows that
  \begin{equation*}
     F(\{\i(\overline{A_2A_3}),\i(\overline{A_3A_4})\})=\{\i'(\overline{A_2'A_3'}),
   \i'(\overline{A_3'A_4'})\}.
  \end{equation*}

   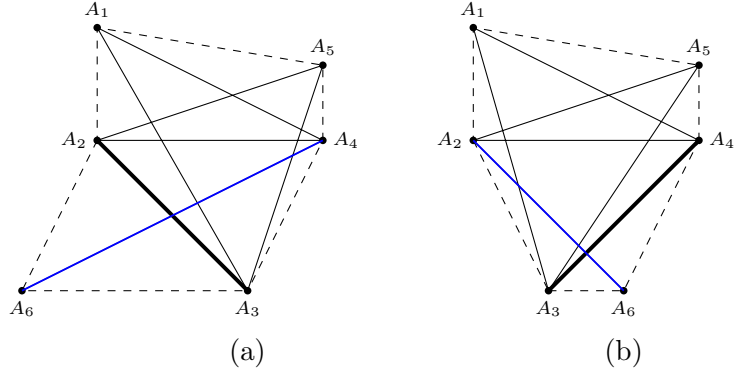
\begin{figure}
   \centering
   \begin{tikzpicture}
     \path[fill,black]
     (0,0.5)coordinate(a5) circle(0.05) node[above]{\tiny$A_1$}
     (3,0)coordinate(a4) circle(0.05) node[above]{\tiny$A_5$}
     (3,-1)coordinate(a3) circle(0.05) node[right]{\tiny$A_4$}
     (2,-3)coordinate(a2) circle(0.05) node[below]{\tiny$A_3$}
     (0,-1)coordinate(a1) circle(0.05) node[left]{\tiny$A_2$}
     (-1,-3)coordinate(a6) circle(0.05) node[below]{\tiny$A_6$};
     \draw[dashed] (a1)--(a2)--(a3)--(a4)--(a5)--cycle  (a1)--(a6)--(a2);
      \draw(a1)--(a3)   (a3)--(a5)--(a2)--(a4)
       (a6)--(a3)  (a4)--(a1);
       \draw[line width=1.5pt] (a1)--(a2) ;
       \draw[line width=0.7pt,blue]  (a3)--(a6);
      \draw (2,-3.5)node[below]{(a)};

        \path[fill,black]
     (0+5,0.5)coordinate(a5) circle(0.05) node[above]{\tiny$A_1$}
     (3+5,0)coordinate(a4) circle(0.05) node[above]{\tiny$A_5$}
     (3+5,-1)coordinate(a3) circle(0.05) node[right]{\tiny$A_4$}
     (1+5,-3)coordinate(a2) circle(0.05) node[below]{\tiny$A_3$}
     (0+5,-1)coordinate(a1) circle(0.05) node[left]{\tiny$A_2$}
     (2+5,-3)coordinate(a6) circle(0.05) node[below]{\tiny$A_6$};
     \draw[dashed] (a1)--(a2)--(a3)--(a4)--(a5)--cycle  (a2)--(a6)--(a3);
      \draw(a1)--(a3)   (a3)--(a5)--(a2)--(a4)
        (a4)--(a1) (a6)--(a1);
       \draw[line width=1.5pt] (a2)--(a3) ;
       \draw[line width=0.7pt,blue]  (a6)--(a1);
      \draw (2+5,-3.5)node[below]{(b)};
   \end{tikzpicture}
   \caption{Identifying edges of strictly convex pentagons I: $\mathrm{d}
   (A_3,\overline{A_{2}A_{4}})=
   \mathcal{D}(\overline{A_{2}A_{4}})$}\label{Fig:pentagon:side1:234}
 \end{figure}
  \vskip5pt
  (ii). If  $\i(\overline{A_1A_2})= \i(\overline{A_4A_5})$, then $\mathcal{C}$ is a simple cylinder. It then follows from Lemma \ref{lem:simplecylinder} and Lemma \ref{lem:pentagon:diagonals} that $F\circ\i(\overline{A_{1}A_{2}})=\i'(\overline{A_{1}'A_{2}'})$ and
  $F\circ\i(\overline{A_{4}A_{5}})=\i'(\overline{A_{4}'A_{5}'})$.
  In the following, we assume that       $\i(\overline{A_1A_2})\neq \i(\overline{A_4A_5})$.
  By Lemma \ref{lem:twostrips}, we can admissibly extend $\i$ to   a hexagon $(A_1A_7A_2A_3A_4A_5)$ such that $(A_1A_7A_2A_4A_5)$ is a  convex pentagon which is strictly convex at $A_7$.  It then follows from Lemma  \ref{lem:quadri:diagonals} that
  \begin{equation}\label{eq:12:1245}
    F\circ\i(\overline{A_1A_2})\in\{\i'(\overline{A_4'A_5'}), \i'(\overline{A_5'A_1'}),
 \i'(\overline{A_1'A_2'}), \i'(\overline{A_2'A_4'})\}.
  \end{equation}
  Since $F\circ\i(\overline{A_2A_4})=\i'(\overline{A_2'A_4'})$, it then follows from Lemma \ref{lem:triangle} that $ F\circ\i(\overline{A_1A_2})\neq \i'(\overline{A_2'A_4'})$. Then (\ref{eq:12:1245}) is reduced to
   \begin{equation}\label{eq:12:12451}
    F\circ\i(\overline{A_1A_2})\in\{\i'(\overline{A_4'A_5'}), \i'(\overline{A_5'A_1'}),
 \i'(\overline{A_1'A_2'})\}.
  \end{equation}

  Let us consider $\overline{A_5A_1}$. There are two subcases: $|\overline{A_5A_1}|>|\overline{A_2A_4}|$ and $|\overline{A_5A_1}|\leq|\overline{A_2A_4}|$.

    {\textbf{Subcase 1.} $|\overline{A_5A_1}|\leq|\overline{A_2A_4}|$.}
     It follows that $\C$ is not a semisimple cylinder which contains $\i(\overline{A_5A_1})$ as a simple boundary component. Therefore, we can admissibly extend $\i|_{(A_1A_2A_4A_5)}$ to a convex pentagon $(A_1A_8A_2A_4A_5)$, (see Figure \ref{Fig:pentagon:side1}(a)),  such that
      \begin{itemize}
        \item $(A_1A_8A_2A_4A_5)$  is strictly convex at $A_2$, and
        \item $\i(A_1A_8A_2A_4A_5)\subset\C$.
      \end{itemize}
      Recall that $\d(A_{5},\overline{A_{2}A_{4}})=
        \d(A_{1},\overline{A_{2}A_{4}})=\mathcal{D}(\overline{A_{2}A_{4}})>
        \d(A_3,\overline{A_{2}A_{4}})$. Then $\i(\inter(A_2A_3A_4))\cap \C=\emptyset$.
        Consequently, $\i:(A_1A_8A_2A_3A_4A_5)\to(X,\omg;\Sigma)$ is admissible (by Lemma \ref{lem:extension1}).
   \begin{itemize}
     \item {\textbf{Subcase 1a.} $F\circ\i(\overline{A_1A_2})=\i'(\overline{A_5'A_1'})$.}
          Let $\Gm$ be a triangulation of $(X,\omg;\Sigma)$ which contains $\i(\overline{A_1A_2})$, $\i(\overline{A_1A_4})$, $\i(\overline{A_2A_4})$, and
     the images of sides of $(A_1A_8A_2A_3A_4A_5)$ by $\i$. Then by Lemma \ref{lem:pentagon:diagonals}, we can also admissibly extend $\i'|_{(A'_1A'_2A'_4A'_5)}$ to a  convex pentagon   ${(A'_1A'_2A'_4A'_5A'_8)}$ which is strictly convex at $A_5'$.
      (see Figure \ref{Fig:pentagon:side1}(b)) such that
  \begin{itemize}
    \item $F\circ\i(\overline{A_8A_4})=
        \i'(\overline{A_8'A_4'})$, and
    \item  $\i'(\overline{A_1'A_8'}), \i'(\overline{A_5'A_8'})\in F(\Gm)$,
  \end{itemize}
   where we use the fact that $\i(\overline{A_4A_8})$ intersects both $\i(\overline{A_1A_2})$ and $\i(\overline{A_2A_5})$ while disjoint from any other saddle connections in $\Gm$ for the first claim.
  Together with   Lemma \ref{lem:extension1} and Lemma \ref{lem:pentagon:diagonals}, this implies  that $\i':(A_1'A_2'A_3'A_4'A_5'A_8')\to(X',\omg';\Sigma')$  is also admissible.
        Then
         $\i'(\overline{A_{8}'A_4'})$ and $\i'(\overline{A_1'A_3'})$ are disjoint on $(X',\omg';\Sigma')$, which contradicts to that $\i(\overline{A_{8}A_4})$ intersects $\i(\overline{A_1A_3})$ on $(X,\omg;\Sigma)$.

     \item{\textbf{Subcase 1b.} $F\circ\i(\overline{A_1A_2})=\i'(\overline{A_4'A_5'})$.} Similarly as in Subcase 1a, we can admissibly extend $\i'$ to a hexagon $(A_1'A_2'A_3'A_4'A_9'A_{5}')$, (see Figure \ref{Fig:pentagon:side1}(c)), such that  $F\circ \i(\overline{A_{8}A_4})=\i'(\overline{A_{9}'A_1'})$. Then
         $\i'(\overline{A_{9}'A_1'})$ and $\i'(\overline{A_1'A_3'})$ are disjoint on $(X',\omg';\Sigma')$, which contradicts to that $\i(\overline{A_{8}A_4})$ intersects $\i(\overline{A_1A_3})$ on $(X,\omg;\Sigma)$.
   \end{itemize}
  Consequently, $F\circ\i(\overline{A_1A_2})=\i'(\overline{A_1'A_2'})$ by (\ref{eq:12:12451}).

\begin{figure}
   \centering
   \begin{tikzpicture}
     \path[fill,black]
     (0.5,0)coordinate(a5) circle(0.05) node[above]{\tiny$A_1$}
     (3,0)coordinate(a4) circle(0.05) node[above]{\tiny$A_5$}
     (3,-2)coordinate(a3) circle(0.05) node[right]{\tiny$A_4$}
     (2,-3)coordinate(a2) circle(0.05) node[below]{\tiny$A_3$}
     (0,-2)coordinate(a1) circle(0.05) node[below]{\tiny$A_2$}
     (-1,0)coordinate(a6) circle(0.05) node[above]{\tiny$A_8$};
     \draw[dashed] (a1)--(a2)--(a3)--(a4)--(a5)--cycle (a1)--(a6)--(a5);
     \draw (a1)--(a3)   (a3)--(a5)--(a2)--(a4)
        (a6)--(a3)  (a4)--(a1);
       \draw[line width=1.5pt] (a1)--(a5) ;
       \draw[line width=0.7pt,blue]  (a3)--(a6);
      \draw (2,-3.5)node[below]{(a)};

        \path[fill,black]
     (0+5,0)coordinate(a5) circle(0.05) node[left]{\tiny $A_1'$}
     (3.5+5,0)coordinate(a4) circle(0.05) node[right]{\tiny$A_5'$　}
     (3+5,-2)coordinate(a3) circle(0.05) node[right]{\tiny$A_4'$ }
     (2+5,-3)coordinate(a2) circle(0.05) node[below]{\tiny$A_3'$}
     (0.3+5,-2)coordinate(a1) circle(0.05) node[left]{\tiny{$A_2'$}}
     (1.5+5,1)coordinate(a6) circle(0.05) node[above]{\tiny{$A_8'$}};
     \draw[dashed] (a1)--(a2)--(a3)--(a4)--(a5)--cycle (a5)--(a6)--(a4);
      \draw (a1)--(a3)  (a2)--(a4) (a3)--(a5) (a4)--(a1) (a5)--(a2)
       (a1)--(a6) (a6)--(a3)  ;
       \draw[line width=1.5pt] (a4)--(a5) ;
       \draw[line width=0.7pt,blue]  (a3)--(a6);
      \draw (2+5,-3.5)node[below]{(b)};

       \path[fill,black]
     (0-1,0-5-1)coordinate(a5) circle(0.05) node[above]{\tiny$A_1'$}
     (3-1,0-5-1)coordinate(a4) circle(0.05) node[above]{\tiny$A_5'$}
     (3-1,-2-5-1)coordinate(a3) circle(0.05) node[below]{\tiny$A_4'$}
     (2-1,-3-5-1)coordinate(a2) circle(0.05) node[below]{\tiny$A_3'$}
     (0-1,-2-5-1)coordinate(a1) circle(0.05) node[below]{\tiny$A_2'$}
     (4-1,0-5-3)coordinate(a6) circle(0.05) node[below]{\tiny$A_9'$};
     \draw[dashed] (a1)--(a2)--(a3)--(a4)--(a5)--cycle (a3)--(a6)--(a4);
     \draw (a1)--(a3)   (a3)--(a5)--(a2)--(a4)
       (a1)--(a6)  (a4)--(a1);
       \draw[line width=1.5pt] (a3)--(a4) ;
       \draw[line width=0.7pt,blue]  (a5)--(a6);
      \draw (2,-3.5-5-1)node[below]{(c)};

        \path[fill,black]
     (0+5,0-6)coordinate(a5) circle(0.05) node[left]{\tiny $A_1$}
     (3.5+5,0-6)coordinate(a4) circle(0.05) node[right]{\tiny$A_5$}
     (3+5,-2-6)coordinate(a3) circle(0.05) node[right]{\tiny$A_4$}
     (2+5,-3-6)coordinate(a2) circle(0.05) node[below]{\tiny$A_3$}
     (0.3+5,-2-6)coordinate(a1) circle(0.05) node[left]{\tiny$A_2$}
     (1.5+5,1-6)coordinate(a6) circle(0.05) node[above]{\tiny$A_9$};
     \draw[dashed] (a1)--(a2)--(a3)--(a4)--(a5)--cycle (a5)--(a6)--(a4);
      \draw (a1)--(a3)  (a2)--(a4) (a3)--(a5) (a4)--(a1) (a5)--(a2)
       (a1)--(a6) (a6)--(a3)  ;
       \draw[line width=1.5pt] (a4)--(a5) ;
       \draw  (a3)--(a6);
      \draw (2+5,-3.5-6)node[below]{(d)};


             \path[fill,black]
     (0,0-6-6)coordinate(a5) circle(0.05) node[above]{\tiny$A_1$}
     (3,0-6-6)coordinate(a4) circle(0.05) node[above]{\tiny$A_5$}
     (3,-2-6-6)coordinate(a3) circle(0.05) node[right]{\tiny$A_4$}
     (2,-3-6-6)coordinate(a2) circle(0.05) node[below]{\tiny$A_3$}
     (0.5,-2-6-6)coordinate(a1) circle(0.05) node[below]{\tiny$A_2$}
     (-1,-2-6-6)coordinate(a6) circle(0.05) node[below]{\tiny$A_{10}$};
     \draw[dashed] (a1)--(a2)--(a3)--(a4)--(a5)--cycle (a1)--(a6)--(a5);
     \draw (a1)--(a3)   (a3)--(a5)--(a2)--(a4)
         (a4)--(a1);
       \draw[line width=1.5pt] (a1)--(a5) ;
       \draw[line width=0.7pt,blue]  (a4)--(a6);
      \draw (2,-3.5-6-6)node[below]{(e)};

   \end{tikzpicture}
   \caption{Identifying edges of strictly convex pentagons II: $\mathcal{I}(A_1A_2A_4A_5)$ is  contained in a cylinder whose boundary contains $\mathcal{I}(\overline{A_2A_4})$
}\label{Fig:pentagon:side1}
 \end{figure}
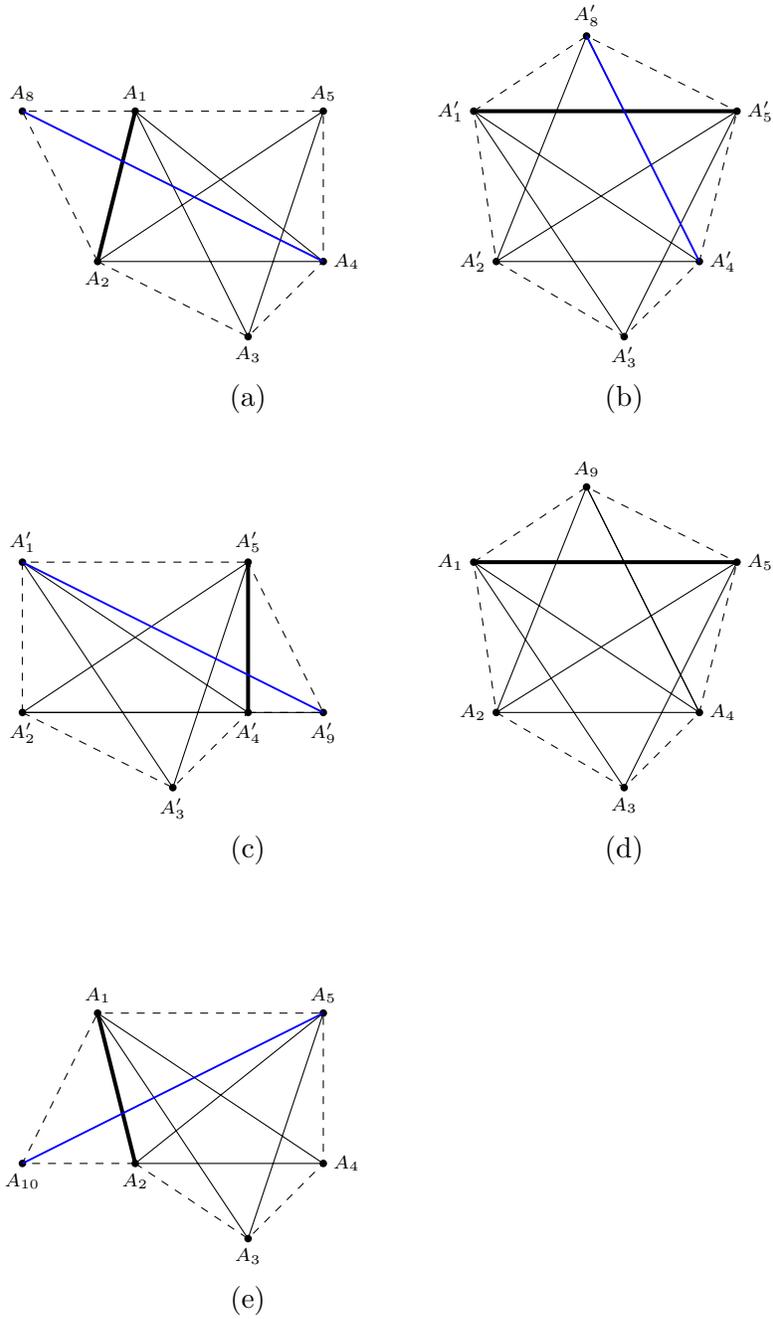

  {\textbf{Subcase 2. $|\overline{A_5A_1}|>|\overline{A_2A_4}|$.}} Consider the strip in the direction of $\overline{A_4A_5}$ which contains $A_1,A_4,A_5$ in the boundary.  It then follows from Lemma \ref{lem:emdedding:strip} that $\i: (A_1A_2A_3A_4A_5)\to(X,\omg;\Sigma)$ can be admissibly extended to a strictly convex hexagon $(A_1A_2A_3A_4A_5A_9)$ (see Figure \ref{Fig:pentagon:side1}(d)). Then $\i(\overline{A_9A_2})$ intersects $\i(\overline{A_1A_4})$ while disjoint from both $\i(\overline{A_3A_5})$ and $\i(\overline{A_2A_4})$. Correspondingly,
   $F\circ\i(\overline{A_9A_2})$ intersects $\i'(\overline{A_1'A_4'})$ while disjoint from both $\i(\overline{A_3'A_5'})$ and $\i(\overline{A_2'A_4'})$. This implies that
      \begin{equation}\label{eq:5112}
         F\circ\i(\overline{A_5A_1})\in \{\i'(\overline{A_5'A_1'}), \i'(\overline{A_1'A_2'})\}.
      \end{equation}
    Similarly, by considering     $\i(\overline{A_9A_4})$ instead of  $\i(\overline{A_9A_2})$, we see that
    \begin{equation}\label{eq:5145}
         F\circ\i(\overline{A_5A_1})\in \{\i'(\overline{A_5'A_1'}), \i'(\overline{A_4'A_5'})\}.
      \end{equation}
      There are two more subcases depending on whether $\i'(\overline{A_1'A_2'})=\i'(\overline{A_4'A_5'})$ or not.
      \begin{itemize}
        \item {\textbf{Subcase 2a.}   $\i'(\overline{A_1'A_2'})=\i'(\overline{A_4'A_5'})$.} Then by Proposition \ref{lem:simplecylinder}, we see that
  $\i(\overline{A_1A_2})= F^{-1}\circ(\i'(\overline{A_1'A_2'}))= F^{-1}\circ(\i'(\overline{A_4'A_5'}))
  = \i(\overline{A_4A_5}).$
  Therefore, $ \i(\overline{A_2A_4})=\i(\overline{A_5A_1})$, which contradicts the assumption that $|\overline{A_5A_1}|>|\overline{A_2A_4}|$.
      \item {\textbf{Subcase 2b.} $\i'(\overline{A_1'A_2'})\neq\i'(\overline{A_4'A_5'})$.} Then
      \begin{equation}\label{eq:5151}
         F\circ\i(\overline{A_5A_1})=\i'(\overline{A_5'A_1'}).
      \end{equation}

      Since $|\overline{A_5A_1}|>|\overline{A_2A_4}|$, it follows that $\C$ is not a semisimple cylinder which contains $\i(\overline{A_2A_4})$ as a simple boundary component. Similarly as in subcase 1a, we can admissibly extend $\i|_{(A_1A_2A_3A_4A_5)}$ to a hexagon $(A_1A_{10}A_2A_3A_4A_5)$ such that $(A_1A_{10}A_2A_4A_5)$ is  convex and strictly convex at $A_1$ (see Figure \ref{Fig:pentagon:side1}(e)).
        Notice that $\i(\overline{A_{10}A_5})$ intersects $\i(\overline{A_1A_4})$ while disjoint from $\i(\overline{A_2A_4})$ and $\i(\overline{A_3A_5})$. Correspondingly, $F\circ\i(\overline{A_{10}A_5})$ intersects $\i(\overline{A_1'A_4'})$ while disjoint from $\i(\overline{A_2'A_4'})$ and $\i(\overline{A_3'A_5'})$.  Therefore, (\ref{eq:12:1245}) is reduced to
         \begin{equation}\label{eq:1251}
        F\circ\i(\overline{A_1A_2})\in\{\i'(\overline{A_1'A_2'}), \i'(\overline{A_5'A_1'})\}.
        \end{equation}
        Combined with (\ref{eq:5151}), this implies that
        $F\circ\i(\overline{A_1A_2})=\i'(\overline{A_1'A_2'})$.
      \end{itemize}

 Similarly, it can be shown that
 $F\circ\i(\overline{A_4A_5})=\i'(\overline{A_4'A_5'})$.
  \vskip 5pt

 (iii).  By Lemma \ref{lem:twostrips},  we can admissibly extend $\i$ to a hexagon $$(A_1A_{11}A_2A_3A_4A_5)$$ such that $(A_1A_{11}A_2A_4A_5)$ is a strictly convex pentagon. The remaining of the proof is similar to that of subcase 1a and subcase 1b in case (ii).

  (iv). By interchanging the labels of $A_2$ and $A_4$,  $A_1$ and $A_5$, we see that this case is equivalent to the case (iii).

 \vskip 5pt

\end{proof}

 If  $(A_1A_2A_3A_4 A_5)$ is not strictly convex at $A_1$, we need to modify Lemma \ref{lem:pentagon:side1}. In this case, by interchanging the labels $A_2$ and $A_5$, $A_3$ and $A_4$, we see that  $\overline{A_1A_3}$ and $\overline{A_1A_4}$ are equivalent, $\overline{A_2A_4}$ and $\overline{A_3A_5}$ are equivalent. But $\overline{A_1A_3}$ and $\overline{A_2A_4}$ are not equivalent in general.

\begin{lemma}[Pentagon lemma III]\label{lem:pentagon:side12}
      Let $\i:(A_1A_2A_3A_4 A_5)\to (X,\omg;\Sigma)$  and    $\i':(A'_1A_2'A_3'A_4'A_5')\to (X',\omg';\Sigma')$ be as in  Lemma \ref{lem:pentagon:diagonals}. Suppose that $(A_1A_2A_3A_4 A_5)$ is not strictly convex at $A_1$.
     \begin{itemize}
       \item ($\overline{A_2A_4}$ and $\overline{A_3A_5}$)
  \begin{enumerate}
     \item[\textup{(i)}] If $\d(A_3,\overline{A_{2}A_{4}})=\mathcal{D}(\overline{A_{2}A_{4}})$, then $$F(\{\i(\overline{A_2A_{3}}),\i(\overline{A_3A_{4}})\})=
        \{\i'(\overline{A_2'A_{3}'}),\i'(\overline{A_3'A_{4}'})\}.$$
      \item[\textup{(ii)}] If  $\d(A_{5},\overline{A_{2}A_{4}})=\mathcal{D}(\overline{A_{2}A_{4}})$, then $$F\circ\i(\overline{A_{5}A_{4}})=\i'(\overline{A_{5}'A_{4}'}).$$
  \end{enumerate}
  Similar results also hold for $\overline{A_3A_5}$.

  \item ($\overline{A_1A_3}$ and $\overline{A_1A_4}$)
  \begin{enumerate}
  \item[\textup{(iii)}] If $\d(A_2,\overline{A_{1}A_{3}})=\mathcal{D}(\overline{A_{1}A_{3}})$, then \\ $F(\{\i(\overline{A_1A_{2}}),\i(\overline{A_2A_{3}})\})=
        \{\i'(\overline{A_1'A_{2}'}),\i'(\overline{A_2'A_{3}'})\}.$
  \item[\textup{(iv)}]  If  $\d(A_5,\overline{A_{1}A_{3}})=\d(A_4,\overline{A_{1}A_{3}})=
         \mathcal{D}(\overline{A_{1}A_{3}})$, then \\ $\{F\circ\i(\overline{A_{3}A_{4}}),F\circ\i(\overline{A_{1}A_{5}})\}\subset
         \{\i'(\overline{A_{3}'A_{4}'}), \i'(\overline{A_{1}'A_{5}'}), \i'(\overline{A_{4}'A_{5}'})\}.$
    \item[\textup{(v)}]
    If  $\d(A_5,\overline{A_{1}A_{3}})=\mathcal{D}(\overline{A_{1}A_{3}})>
    \d(A_4,\overline{A_{1}A_{3}})$,  then \\ $F\circ\i(\overline{A_{1}A_{5}})=\i'(\overline{A_{1}'A_{5}'}).$
    \item[\textup{(vi)}]
     If  $\d(A_4,\overline{A_{1}A_{3}})=\mathcal{D}(\overline{A_{1}A_{3}})>
     \d(A_5,\overline{A_{1}A_{3}})$, then \\ $F\circ\i(\overline{A_{3}A_{4}})=\i'(\overline{A_{3}'A_{4}'}).$
  \end{enumerate}
  Similar results also hold for $\overline{A_1A_4}$.
   \end{itemize}
\end{lemma}

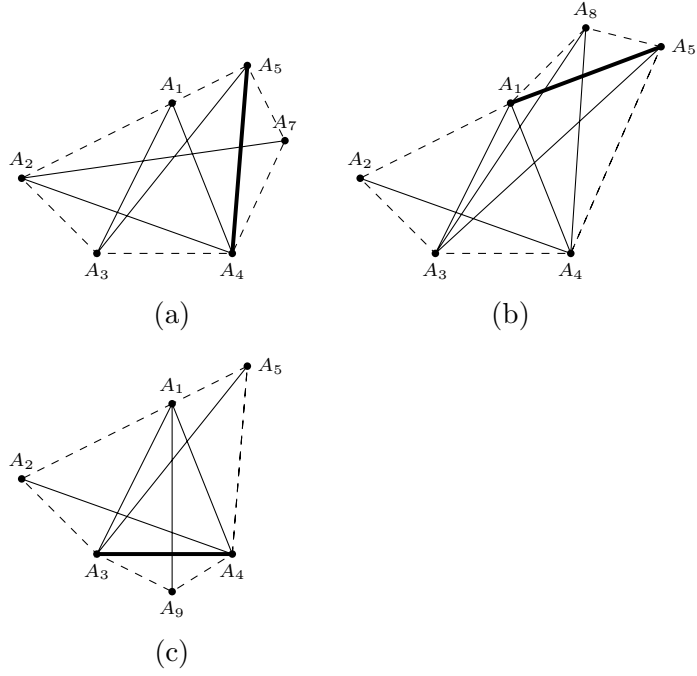
\begin{figure}
   \centering
   \begin{tikzpicture}
     \path[fill,black]
     (0,0)coordinate(a1) circle(0.05) node[above]{\tiny$A_1$}
     (-2,-1)coordinate(a2) circle(0.05) node[above]{\tiny$A_2$}
     (1,0.5)coordinate(a5) circle(0.05) node[right]{\tiny$A_5$}
     (-1,-2)coordinate(a3) circle(0.05) node[below]{\tiny$A_3$}
     (0.8,-2)coordinate(a4) circle(0.05) node[below]{\tiny$A_4$}
     (1.5,-0.5)coordinate(a7) circle(0.05) node[above]{\tiny$A_7$};
     \draw[dashed] (a1)--(a2)--(a3)--(a4)--(a7)--(a5)--cycle (a4)--(a5);
     \draw (a1)--(a3)--(a5) (a1)--(a4)--(a2)--(a7);
       \draw[line width=1.5pt] (a4)--(a5) ;
      \draw (0,-2.5)node[below]{(a)};

        \path[fill,black]
     (0+4.5,0)coordinate(a1) circle(0.05) node[above]{\tiny$A_1$}
     (-2+4.5,-1)coordinate(a2) circle(0.05) node[above]{\tiny$A_2$}
     (1.5+5,0.75)coordinate(a5) circle(0.05) node[right]{\tiny$A_5$}
     (-1+4.5,-2)coordinate(a3) circle(0.05) node[below]{\tiny$A_3$}
     (0.8+4.5,-2)coordinate(a4) circle(0.05) node[below]{\tiny$A_4$}
     (1+4.5,1)coordinate(a8) circle(0.05) node[above]{\tiny$A_8$};
     \draw[dashed] (a1)--(a2)--(a3)--(a4)--(a5)--(a8)--cycle (a4)--(a5);
     \draw (a1)--(a3)--(a5) (a1)--(a4)--(a2)  (a3)--(a8)--(a4);
       \draw[line width=1.5pt] (a1)--(a5) ;
      \draw (0+4.5,-2.5)node[below]{(b)};

      \path[fill,black]
     (0,0-4)coordinate(a1) circle(0.05) node[above]{\tiny$A_1$}
     (-2,-1-4)coordinate(a2) circle(0.05) node[above]{\tiny$A_2$}
     (1,0.5-4)coordinate(a5) circle(0.05) node[right]{\tiny$A_5$}
     (-1,-2-4)coordinate(a3) circle(0.05) node[below]{\tiny$A_3$}
     (0.8,-2-4)coordinate(a4) circle(0.05) node[below]{\tiny$A_4$}
     (0,-2.5-4)coordinate(a9) circle(0.05) node[below]{\tiny$A_9$};
     \draw[dashed] (a1)--(a2)--(a3)--(a9)--(a4)--(a5)--cycle (a4)--(a5);
     \draw (a1)--(a3)--(a5) (a1)--(a4)--(a2)  (a1)--(a9);
       \draw[line width=1.5pt] (a3)--(a4) ;
      \draw (0,-3-4)node[below]{(c)};

   \end{tikzpicture}
   \caption{Identifying edges of non-strictly convex pentagons.}\label{Fig:pentagon:side12}
 \end{figure}

\begin{proof}
   The proofs of cases (i) and (iii) are similar to the proof of case (i) in Lemma \ref{lem:pentagon:side1}.
   It remains to consider (ii), (iv), (v), and (vi).

   \vskip5pt
  (ii). Notice that  $(A'_1A'_2A'_3A'_4A'_5)$ is   strictly convex at $A_2',A_3',A_4',A_5'$. If $(A'_1A'_2A'_3A'_4A'_5)$ is  also strictly convex at $A_1'$, then it is a strictly convex pentagon. It then follows from the first conclusion of Lemma \ref{lem:pentagon:diagonals} that $(A_1A_2A_3A_4A_5)$ is also a strictly convex pentagon, which contradicts the assumption that  $(A_1A_2A_3A_4A_5)$ is not  strictly convex at $A_1$. Consequently, $(A'_1A'_2A'_3A'_4A'_5)$ is  not strictly convex at $A_1'$.

 By Lemma \ref{lem:twostrips}, we can admissibly extend $\i$ to a hexagon $(A_1A_2A_3A_4A_7A_5)$ such that $(A_1A_2A_4A_7A_5)$ is convex and strictly convex at$A_2$, $A_4,A_7, A_5$ (see Figure \ref{Fig:pentagon:side12}(a)). Let $\Gm_1$ be a triangulation of $(X,\omg;\Sigma)$ which contains $\i(\overline{A_2A_4})$, $\i(\overline{A_1A_4})$, $\i(\overline{A_5A_4})$, and the images of all sides of
  $(A_1A_2A_3A_4A_7A_5)$ by $\i$. By Lemma \ref{lem:pentagon:diagonals}, we see that $F(\Gm_1)$ contains $\i(\overline{A_2'A_4'})$, $\i(\overline{A_1'A_4'})$, $F\circ\i(\overline{A_5A_4})$  and all images of sides of $(A_1'A_2'A_3'A_4'A_5')$ by $\i'$. Notice that $\i(\overline{A_2A_7})$ intersects both $\i(\overline{A_1A_4})$ and $\i(\overline{A_5A_4})$ while disjoint from any other saddle connections in $\Gm_1$. Correspondingly, $F\circ\i(\overline{A_2A_7})$ intersects both $\i'(\overline{A_1'A_4'})$ and $F\circ\i(\overline{A_5A_4})$ while disjoint from any other saddle connections in $F(\Gm_1)$. This implies that
  $$ F\circ\i(\overline{A_5A_4})\in\{\i'(\overline{A_5'A_4'}), \i'(\overline{A_5'A_1'}), \i'(\overline{A_1'A_2'})\}. $$
  Since $\i(\overline{A_1A_7})$ intersects $\i(\overline{A_5A_4})$ while disjoint from both $\i(\overline{A_1A_3})$ and $\i(\overline{A_2A_4})$, it follows that $F\circ\i(\overline{A_1A_7})$ intersects $F\circ\i(\overline{A_5A_4})$ while disjoint from both $\i'(\overline{A_1'A_3'})$ and $\i'(\overline{A_2'A_4'})$. This implies that
  $  F\circ\i(\overline{A_5A_4})\neq \i'(\overline{A_1'A_2'}).$
  If $ F\circ\i(\overline{A_5A_4})= \i'(\overline{A_5'A_1'})$, then by Lemma \ref{lem:pentagon:diagonals}, we can admissibly extend $\i'|_{(A_1'A_2'A_4'A_5')}$ to a  pentagon $(A_1'A_2'A_4'A_5'A_7')$ such that $F\circ\i(\overline{A_2A_7})=\i'(\overline{A_2'A_7'})$ (since $\i(\overline{A_2A_7})$ intersects both $\i(\overline{A_1A_4})$ and $\i(\overline{A_5A_4})$ ).
   This implies that $(A_1'A_2'A_4'A_5')$  is strictly convex at $A_1'$, which contradicts to the assumption that  $(A'_1A'_2A'_3A'_4A'_5)$ is also not strictly convex at $A_1'$.
  Consequently, $ F\circ\i(\overline{A_5A_4})= \i'(\overline{A_5'A_4'})$.

  (iv) If $\i(\overline{A_3A_4})=\i(\overline{A_5A_1})$, then the claim follows from Proposition \ref{lem:simplecylinder}. Otherwise, it follows directly from Lemma \ref{lem:twostrips}, Lemma \ref{lem:quadri:diagonals} and Lemma \ref{lem:triangle}.

  (v) It follows from Lemma \ref{lem:twostrips} that we can admissibly extend $\i$ to a hexagon $(A_1A_2A_3A_4A_5A_8)$ such that $(A_1A_3A_4A_5A_8)$ is a strictly pentagon (see Figure \ref{Fig:pentagon:side12}(b)). It then follows Lemma  \ref{lem:quadri:diagonals} and Lemma \ref{lem:triangle} that
  $$ F\circ\i(\overline{A_5A_1})\in\{\i'(\overline{A_3'A_4'}),\i'(\overline{A_4'A_5'}),
  \i'(\overline{  A_5'A_1'})\}. $$
  On the other hand, $\i(\overline{A_4A_8})$ intersects $\i(\overline{A_5A_1})$ while disjoint from $\i(\overline{A_1A_3})$ and $\i(\overline{A_2A_4})$. Therefore,
  $$ F\circ\i(\overline{A_5A_1})\in\{\i'(\overline{A_4'A_5'}),
  \i'(\overline{  A_5'A_1'})\}. $$
  If $F\circ\i(\overline{A_5A_1})=\i'(\overline{A_4'A_5'})$, it follows from Lemma \ref{lem:pentagon:diagonals} and Lemma \ref{lem:extension1} that $\i'$ can be admissibly extended to a hexagon $(A_1'A_2'A_3'A_4'A_8'A_5')$ such that $(A_1'A_3'A_4'A_8'A_5')$ is a strictly convex pentagon. Therefore, $(A_1'A_2'A_3'A_4'A_8'A_5')$ is a convex hexagon strictly convex at $A_2',A_3',A_4',A_8',A_5'$.
   Let $\Gm_2$ be a triangulation of $(X,\omg;\Sigma)$ which contains
   $\i(\overline{A_1A_3})$, $\i(\overline{A_1A_4})$, $\i(\overline{A_1A_5})$ and all images of sides of $(A_1A_2A_3A_4A_5A_8)$ by $\i$. Then by Lemma \ref{lem:triangulation:correspondence}  and  Lemma \ref{lem:pentagon:diagonals},  $F(\Gm_2)$ is a triangulation of $(X',\omg';\Sigma')$ which contains
   $\i'(\overline{A_1'A_3'})$, $\i'(\overline{A_1'A_4'})$, $\i'(\overline{A_4'A_5'})$ and all images of sides of $(A'_1A'_2A'_3A'_4A'_8A_5')$ by $\i'$. Notice that $\i'(\overline{A_2'A_8'})$ intersects each of $\{\i'(\overline{A_1'A_3'}),\i'(\overline{A_1'A_4'}),
   \i'(\overline{A_4'A_5'})\}$ while disjoint from any other saddle connection in $F(\Gm_2)$.
   Therefore, $F^{-1}\circ\i'(\overline{A_2'A_8'})$ intersects each of $\{\i(\overline{A_1A_3}),\i(\overline{A_1A_4}),
   \i(\overline{A_4A_5})\}$ while disjoint from any other saddle connection in $\Gm_2$. This implies that $(A_1A_2A_4A_5A_8)$ is strictly convex at $A_1$. In particular, $(A_1A_2A_4A_5)$ is strictly convex at $A_1$, which contradicts the assumption that $(A_1A_2A)3A_4A_5)$ is not strictly convex at $A_1$.

  (vi) It follows from Lemma \ref{lem:twostrips} that we can admissibly extend $\i$ to a hexagon $(A_1A_2A_3A_9A_4A_5)$ such that $(A_1A_3A_9A_4A_5)$ is a strictly convex pentagon. It then follows Lemma \ref{lem:quadri:diagonals} and Lemma \ref{lem:triangle} that
  $$ F\circ\i(\overline{A_3A_4})\in\{\i'(\overline{A_3'A_4'}),\i'(\overline{A_4'A_5'}),
  \i'(\overline{  A_5'A_1'})\}. $$
  Notice that $\i(\overline{A_1A_9})$ intersects each of $\i(\overline{A_2A_4}),\i(\overline{A_3A_5}),\i(\overline{A_3A_4})$. Therefore,
  $F\circ\i(\overline{A_3A_4})=\i'(\overline{A_3'A_4'})$. Otherwise, $F\circ\i(\overline{A_1A_9})$ would disjoint from $\i'(\overline{A_2'A_4'})$ or $\i'(\overline{A_3'A_5'})$.
\end{proof}

Based on Lemma \ref{lem:pentagon:side1} and Lemma \ref{lem:pentagon:side12}, we now prove the following.
\begin{lemma}[Pentagon lemma IV]\label{lem:pentagon:side2}
  Let $\i:(A_1A_2A_3A_4 A_5)\to (X,\omg;\Sigma)$ and $\i':(A'_1A_2'A_3'A_4'A_5')\to(X',\omg';\Sigma')$ be as  in Lemma \ref{lem:pentagon:diagonals}. Then there exist two adjacent sides of $(A_1A_2A_3A_4A_5)$, say $\overline{A_{k-1}A_k}$ and $\overline{A_kA_{k+1}}$, such that $F\circ\i(\overline{A_{k-1}A_k})=\i'(\overline{A'_{k-1}A'_k})$ and $F\circ\i(\overline{A_{k}A_{k+1}})=\i'(\overline{A'_{k}A'_{k+1}})$.
Moreover, if $(A_1A_2A_3A_4A_5)$ is not strictly convex at $A_1$, then $k$ can be chosen to be different from 1.
\end{lemma}

\begin{table}[h]
   \linespread{1.5}
   \begin{center}
  \begin{tabular}{|c|c|c|c|}
   \hline
    \textbf{Diagonals} & \textbf{Case 1} & \textbf{Case 2} & \textbf{Case 3}\\ [3pt] \hline
    $ \overline{A_2A_4}$ & $F(\gm_1)=\gm_1'$ & $F(\gm_4)=\gm_4'$ &  $F(\{\gm_2,\gm_3\})=\{\gm_2',\gm_3'\}$ \\[3pt] \hline
     $ \overline{A_3A_5}$ & $F(\gm_2)=\gm_2'$ & $F(\gm_5)=\gm_5'$ &  $F(\{\gm_3,\gm_4\})=\{\gm_3',\gm_4'\}$ \\[3pt] \hline
   $ \overline{A_1A_3}$ & $F(\gm_5)=\gm_5'$ & $F(\gm_3)=\gm_3'$ &  $F(\{\gm_1,\gm_2\})=\{\gm_1',\gm_2'\}$ \\ [3pt]  \hline
   $ \overline{A_4A_1}$ & $F(\gm_3)=\gm_3'$ & $F(\gm_1)=\gm_1'$ &  $F(\{\gm_4,\gm_5\})=\{\gm_4',\gm_5'\}$ \\[3pt] \hline
   $ \overline{A_5A_2}$ & $F(\gm_4)=\gm_4'$ & $F(\gm_2)=\gm_2'$ &  $F(\{\gm_1,\gm_5\})=\{\gm_1',\gm_5'\}$ \\[3pt] \hline
  \end{tabular}
 \\ \smallskip (a)
 \vskip 10pt
   \begin{tabular}{l|c||l|c||l|c}
   \noalign{\smallskip}\hline\noalign{\smallskip}
     \multirow{2}*{\textbf{Case}} & \textbf{Sides} & \multirow{2}*{\textbf{Case}} & \textbf{Sides} &  \multirow{2}*{\textbf{Case}} & \textbf{Sides}  \\
     &\textbf{Identified}&&\textbf{Identified}&&\textbf{Identified}\\ \hline
     (111) & $\gm_1,\gm_2,\gm_5$ & (112) & $\gm_1,\gm_2,\gm_3$ &(113) &$\gm_1,\gm_2$\\
     (121) & $\gm_5,\gm_1$ & (122) & $\gm_5,\gm_1,\gm_3$ &(123) &$\gm_5,\gm_1,\gm_2$\\
     (131) & $\gm_5,\gm_1$ & (132) & $\gm_1,\gm_3,\gm_4$ &(133) &$\gm_1,\gm_2$\\
     (211) & $\gm_2,\gm_4,\gm_5$ & (212) & $\gm_2,\gm_3,\gm_4$ &(213) &$\gm_1$,$\gm_2$,$\gm_4$\\
     (221) & $\gm_4,\gm_5$ & (222) & $\gm_3,\gm_4,\gm_5$ &(223) &$\gm_4,\gm_5$\\
     (231) & $\gm_3,\gm_4,\gm_5$ & (232) & $\gm_3,\gm_4$ &(233) &$\gm_3,\gm_4$\\
     (311) & $\gm_2,\gm_3,\gm_5$ & (312) & $\gm_2,\gm_3$ &(313) &$\gm_1,\gm_2,\gm_3$\\
     (321) & $\gm_5$ & (322) & $\gm_2,\gm_3,\gm_5$ &(323) &$\gm_1,\gm_2,\gm_3,\gm_5$\\
     (331) & $\gm_2,\gm_3,\gm_4,\gm_5$ & (332) & $\gm_2,\gm_3,\gm_4$ &(333) &$\gm_1,\gm_2,\gm_3,\gm_4$\\
     \hline
   \end{tabular}
   \\ \smallskip (b)
   \end{center}
   \caption{\small{In table (a), we list all three possibilities for each diagonal. In table (b), we list all 27 possibilities for  $\overline{A_2A_4}$, $\overline{A_3A_5}$ and $\overline{A_1A_3}$, where the triple $(ijk)$
   represents the possibility corresponding to case $i$ for $\overline{A_2A_4}$,
    case $j$ for  $\overline{A_3A_5}$,  and case $k$ for $\overline{A_1A_3}$.}}\label{tab:cases}
 \end{table}
\begin{proof}
For convenience, let $\gm_i:=\i(\overline{A_iA_{i+1}})$, $\gm_i':=\i(\overline{A'_iA'_{i+1}})$, $i=1,2,3,4,5$, where $A_6=A_1$ and $A_6'=A_1'$. $\gm_i$ is said to be \textit{identified} if $F(\gm_i)=\gm_i'$.
The proof will be split into two cases depending on whether $(A_1A_2A_3A_4A_5)$ is strictly convex at $A_1$.

\vskip10pt
\noindent\textbf{Case I: $(A_1A_2A_3A_4A_5)$ is strictly convex at $A_1$. }

Consider the diagonals $\overline{A_2A_4}$, $\overline{A_3A_5}$ and $\overline{A_1A_3}$. By Lemma \ref{lem:pentagon:side1}, there are three cases for each diagonal as listed in Table \ref{tab:cases}(a). Therefore, there are 27 possibilities in total. Let us denote by the triple $(ijk)$ the possibility corresponding to case $i$ for $\overline{A_2A_4}$, case $j$ for  $\overline{A_3A_5}$,  and case $k$ for $\overline{A_1A_3}$. We list all possibilities in Table \ref{tab:cases}(b), where $\gm_i$ is said to be identified by $F$ if $F(\gm_i)=\gm_i'$. (For cases (331), (332) and (333), we have $F(\{\gm_2,\gm_3\})=\{\gm_2',\gm_3'\}$ and $F(\{\gm_3,\gm_4\})=\{\gm_3',\gm_4'\}$.   To show that all of $\gm_2,\gm_3,\gm_4$ are identified, it suffices to show that $\gm_3$ is identified. Suppose to the contrary that $\gm_3$ is not identified, then $F(\gm_2)=F(\gm_4)=\gm_3'$. Therefore $\gm_2=\gm_4$. It then follows from Proposition \ref{lem:simplecylinder} that  $\i'(A_2'A_3'A_4'A_5')$ is a simple cylinder which contains $\gm_3'$ as an interior saddle connection, which contradicts  the fact that $(A_1'A_2'A_3'A_4'A_5')$ is an admissible strictly convex pentagon.)

We see that for each  possibility $(ijk)\neq (321)$, there exists some $m$ such that $F(\gm_m)=\gm_m'$ and $F(\gm_{m+1})=\gm_{m+1}')$, where $\gm_6=\gm_1$ and $\gm_6'=\gm_1'$. To finish the proof of this case,
it remains to consider the possibility $(321)$. Notice that in this possibility, we have
\begin{equation}\label{eq:213}
  F(\gm_5)=\gm_5', F(\{\gm_2,\gm_3\})=\{\gm_2',\gm_3'\}.
\end{equation}  Let us consider one more diagonal $\overline{A_1A_4}$. There are three subcases.
\begin{itemize}
  \item If $F(\gm_3)=\gm_3'$, then by Equation (\ref{eq:213}), we also have $F(\gm_2)=\gm_2'$. The proof completes.
  \item If $F(\gm_1)=\gm_1'$, the proof completes.
  \item If $F(\{\gm_4,\gm_5\})=\{\gm_4',\gm_5'\}$, then by Equation (\ref{eq:213}), we also have $ F(\gm_4)=\gm_4'$. The proof completes.
\end{itemize}


\begin{table}[h]
   \linespread{1.5}
   \begin{center}
  \begin{tabular}{|c|c|c|c|}
   \hline
    \textbf{Diagonals} & \textbf{Case 1} & \textbf{Case 2} \\ [3pt] \hline
    $ \overline{A_2A_4}$ & $F(\gm_4)=\gm_4'$ &$F(\{\gm_2,\gm_3\})=\{\gm_2',\gm_3'\}$  \\[3pt] \hline
     $ \overline{A_3A_5}$ & $F(\gm_2)=\gm_2'$& $F(\{\gm_3,\gm_4\})=\{\gm_3',\gm_4'\}$ \\[3pt] \hline
  \end{tabular}
 \\ \smallskip (a)
  \vskip10pt
   \begin{tabular}{l|c||l|c}
   \noalign{\smallskip}\hline\noalign{\smallskip}
     \multirow{2}*{\textbf{Case}} & \textbf{Sides} & \multirow{2}*{\textbf{Case}} & \textbf{Sides}  \\
     &\textbf{Identified}&&\textbf{Identified}\\ \hline
     (11) & $\gm_2,\gm_4$ & (12) & $\gm_3,\gm_4$ \\
     (21) & $\gm_2,\gm_3$ & (22) & $\gm_2,\gm_3,\gm_4$ \\
     \hline
   \end{tabular}
   \\ \smallskip (b)
   \end{center}
   \caption{\small{In table (a), we list all possibilities for diagonals $\overline{A_2A_4}$ and $\overline{A_3A_5}$. In table (b), we list all four possibilities for  the combination of $\overline{A_2A_4}$, $\overline{A_3A_5}$, where the pair $(ij)$
   represents the possibility corresponding to case $i$ for $\overline{A_2A_4}$ and
    case $j$ for  $\overline{A_3A_5}$.}}\label{tab:cases2}
 \end{table}

\vskip 10pt
\noindent\textbf{Case II: $(A_1A_2A_3A_4A_5)$ is not strictly convex at $A_1$.}

In this case, for each diagonal, there are two or four cases by Lemma \ref{lem:pentagon:side12} (see Table \ref{tab:cases2}(a))
Consider the diagonals $\overline{A_2A_4}$ and  $\overline{A_3A_5}$. There are 4 possibilities in total. Let us denote by the triple $(ij)$ the possibility corresponding to case $i$ for $\overline{A_2A_4}$ and case $j$ for  $\overline{A_3A_5}$. We list all possibilities in Table \ref{tab:cases2}(c).  For  case (22), we have $F(\{\gm_2,\gm_3\})=\{\gm_2',\gm_3'\}$ and $F(\{\gm_3,\gm_4\})=\{\gm_3',\gm_4'\}$. To show that all of $\gm_2,\gm_3,\gm_4$ are identified, it suffices to show that $\gm_3$ is identified. Suppose to the contrary that $\gm_3$ is not identified, then $F(\gm_2)=F(\gm_4)=\gm_3'$. Hence $\d(A_5,\overline{A_2A_4})=\mathcal{D}(\overline{A_2A_4})$. It then follows from the second statement of Lemma \ref{lem:pentagon:side12} that $\gm_4$ is identified. This is a contradiction which proves that   all of $\gm_2,\gm_3,\gm_4$ are identified.

 We see that for each  possibility $(ij)\neq (11)$, there exists some $m$ such that $F(\gm_m)=\gm_m'$ and $F(\gm_{m+1})=\gm_{m+1}')$, where $\gm_6=\gm_1$ and $\gm_6'=\gm_1'$. To finish the proof of this case,
it remains to consider the possibility $(11)$. Notice that in this possibility, we have
\begin{equation}\label{eq:11}
  F(\gm_2)=\gm_2', F(\gm_4)=\gm_4'.
\end{equation}  Let us consider the diagonal $\overline{A_1A_3}$. By Lemma \ref{lem:pentagon:side12},there are four cases.
 \begin{itemize}
  \item[  (\textbf{C1})] $F(\gm_5)=\gm_5'$;
  \item [ (\textbf{C2})] $F(\gm_3)=\gm_3'$; 
  \item[  (\textbf{C3})]$F(\{\gm_1,\gm_2\})=\{\gm_1',\gm_2'\}$; 
  \item[  (\textbf{C4})]$\{F(\gm_3),F(\gm_5)\}\subset\{\gm_3',\gm_4',\gm_5'\}$, $\overline{A_1A_3}$ and $\overline{A_4A_5}$ are parallel.
    \end{itemize}
    If  $F(\gm_3)=\gm_3'$ or $F(\gm_5)=\gm_5'$, then the proof completes. If $F(\{\gm_1,\gm_2\})=\{\gm_1',\gm_2'\}$, then by Equation (\ref{eq:11}), it follows that $ F(\gm_1)=\gm_1'$. The proof also completes. If  $\{F(\gm_3),F(\gm_5)\}\subset\{\gm_3',\gm_4',\gm_5'\}$, together with (\ref{eq:11}), this  reduces to
   \begin{equation*}
  \{F(\gm_3),F(\gm_5)\}=\{\gm_3',\gm_5'\}.
  \end{equation*}
  If $F(\gm_3)=\gm_3'$ and $F(\gm_5)=\gm_5'$, then the proof completes. It remains to consider the case
  \begin{equation}\label{eq:35}
    F(\gm_3)=\gm_5',~F(\gm_5)=\gm_3'.
  \end{equation}

   Similarly, let us consider the   diagonal $\overline{A_1A_4}$. By Lemma \ref{lem:pentagon:side12}, there are four more possibilities.
    \begin{itemize}
  \item[ (\textbf{C5})]  $F(\gm_3)=\gm_3'$;
  \item[ (\textbf{C6})] $F(\gm_1)=\gm_1'$;
  \item[ (\textbf{C7})] $F(\{\gm_4,\gm_5\})=\{\gm_4',\gm_5'\}$;
  \item[ (\textbf{C8})] $\{F(\gm_1),F(\gm_3)\}\subset\{\gm_1',\gm_2',\gm_3'\}$, $\overline{A_1A_4}$ and $\overline{A_2A_3}$ are parallel.
  \end{itemize}
 If  $F(\gm_3)=\gm_3'$, $F(\gm_1)=\gm_1'$, or $F(\{\gm_4,\gm_5\})=\{\gm_4',\gm_5'\}$, then the proof completes as discussed above. If $\{F(\gm_1),F(\gm_3)\}\subset\{\gm_1',\gm_2',\gm_3'\}$, together with (\ref{eq:11}), this reduces to
   \begin{equation*}
  \{F(\gm_1),F(\gm_3)\}=\{\gm_1',\gm_3'\}.
  \end{equation*}
 If $F(\gm_3)=\gm_3'$ and $F(\gm_1)=\gm_1'$, the proof completes. It remains to consider the case
 \begin{equation}\label{eq:13}
   F(\gm_1)=\gm_3',~F(\gm_3)=\gm_1'.
 \end{equation}

 Next, let us consider the case  where both (\textbf{C4}) and (\textbf{C8}) hold.  In this case, we have $\gm_5=F^{-1}(\gm_3')=\gm_1$.  Recall that $\overline{A_1A_3}$ and $\overline{A_4A_5}$ are parallel for   (\textbf{C4}),  and that  $\overline{A_1A_4}$ and $\overline{A_2A_3}$ are parallel for  (\textbf{C8}). Then $$\d(A_2,\overline{A_{1}A_{3}})=\d(A_5,\overline{A_{1}A_{3}})
 =\mathcal{D}(\overline{A_{1}A_{3}}), $$
which reduces this case to case (\textbf{C3}). The proof completes.

\end{proof}

\subsection{Quadrilaterals}

   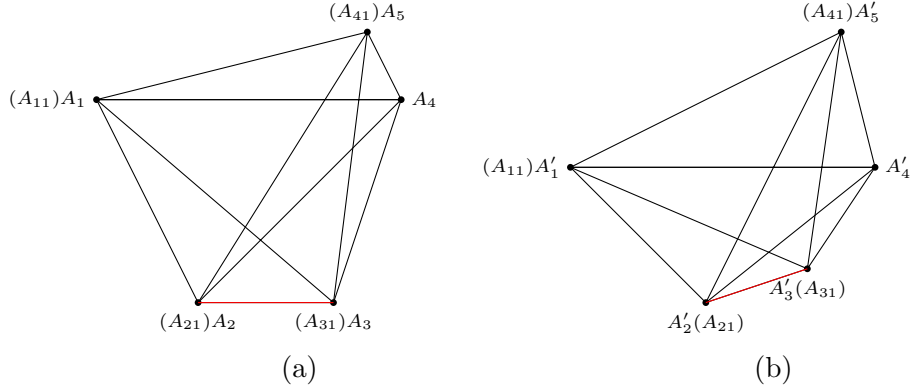
\begin{figure}[t]
  \begin{tikzpicture}[scale=0.9]


       \path[fill,black]
   (3+2-9,-3-6) coordinate(c2) circle(0.05) node[left]{\tiny $(A_{11})A_1$}
   (6+3-9,-3-6+1) coordinate(c1) circle(0.05)node[above]{\tiny $(A_{41})A_5$}
   (9-9.5,-6-6) coordinate(c4) circle(0.05)node[below]{\tiny $(A_{31})A_3$}
   (7.5+2-9,-3-6) coordinate(c5) circle(0.05)node[right]{\tiny $A_4$}
   (4.5+2-9,-5-1-6) coordinate(c3) circle(0.05)node[below]{\tiny $(A_{21})A_2$};
   \draw (c1)--(c2)--(c3)--(c4)--(c5)--cycle;
   \draw (c3)--(c1)--(c4) (c2)--(c5);
   \draw(c2)--(c4) (c3)--(c5)  (c2)--(c5);
   \draw[red](c4)--(c3);
   \draw
          (-3+2,-7-6)node{(a)};


           \path[fill,black]
   (3+2-9+8-1,-3-6-1) coordinate(c2) circle(0.05) node[left]{\tiny ($A_{11}$)$A'_1$}
   (6+3-9+8-1,-3-6+1) coordinate(c1) circle(0.05)node[above]{\tiny ($A_{41}$)$A'_5$}
   (9-9.5+8-1,-6-5.5) coordinate(c4) circle(0.05)node[below]{\tiny $A'_3$($A_{31}$)}
   (7.5+2-9+8-1,-3-6-1) coordinate(c5) circle(0.05)node[right]{\tiny $A'_4$}
   (4.5+2.5-9+8-1,-5-1-6) coordinate(c3) circle(0.05)node[below]{\tiny $A'_2$($A_{21}$)};
   \draw (c1)--(c2)--(c3)--(c4)--(c5)--cycle;
   \draw (c3)--(c1)--(c4) (c2)--(c5);
   \draw(c2)--(c4) (c3)--(c5)  (c2)--(c5);
   \draw[red](c4)--(c3);
   \draw
          (-3+2+8-1,-7-6)node{(b)};
  \end{tikzpicture}
  \caption{ Caption to Figure 10:
Identifying edges of strictly convex quadrilaterals.}\label{Fig:quadriletaral:34}
\end{figure}

 A direct consequence of Lemma \ref{lem:pentagon:side2} is that we can improve Lemma \ref{lem:quadri:diagonals}.
\begin{lemma}\label{cor:3to4}
  Let $\i:(A_1A_2A_3A_4)\to(X,\omg;\Sigma)$ and $\i':(A'_1A_2'A_3'A_4')\to (X',\omg';\Sigma')$ be two admissible maps, where $(A_1A_2A_3A_4)$ and $(A'_1A_2'A_3'A_4')$ are strictly convex quadrilaterals. Suppose that
  \begin{itemize}
    \item  $F\circ\i(\overline{A_1A_3})= \i'(\overline{A_1'A_3'})$,
    $F\circ\i(\overline{A_2A_4})= \i'(\overline{A_2'A_4'})$,
    \item $ F\circ\i(\overline{A_1A_2})= \i'(\overline{A_1'A_2'})$,
    $ F\circ\i(\overline{A_1A_4})=\i'(\overline{A_1'A_4'})$,
 \\  $F\circ\i(\overline{A_3A_4})=\i'(\overline{A_3'A_4'}) $.
  \end{itemize}
  Then $F\circ\i(\overline{A_2A_3})=\i'(\overline{A_2'A_3'})$.
\end{lemma}
\begin{proof}
 If the image  of $(A_1A_2A_3A_4)$ by $\i$ is a simple cylinder, then the lemma follows from Proposition \ref{lem:simplecylinder}. In the following, we assume that the image  of $(A_1A_2A_3A_4)$ by $\i$ is not a simple cylinder. In particular, $\i(\overline{A_1A_2})\neq  \i(\overline{A_3A_4})$ and $\i(\overline{A_1A_4})\neq  \i(\overline{A_2A_3})$.

  Let $\theta_i$ be the interior angle of $(A_1A_2A_3A_4)$ at $A_i$, $i=1,2,3,4$.
 If $\theta_2+\theta_3\leq\pi$, it then  follows from  Lemma \ref{lem:emdedding:strip} and Lemma \ref{lem:quadri:diagonals} that
 $F\circ\i(\overline{A_2A_3})=\i'(\overline{A'_{i-1}A_{i}'})$ for some $i=1,2,3,4$, where $A'_0=A_4'$. Since
 $\i(\overline{A_1A_4})\neq  \i(\overline{A_2A_3})$, it follows that
 $F\circ\i(\overline{A_2A_3})\neq \i'(\overline{A'_1A_{4}'})$. It then follows from Lemma \ref{lem:triangle} that $F\circ\i(\overline{A_2A_3})= \i'(\overline{A'_2A_{3}'})$.

   \vskip 5pt
   Suppose that $\theta_2+\theta_3>\pi$. Consider the algorithm below.
   \begin{quotation}

       Since  $\theta_2+\theta_3>\pi$, then  $\theta_1+\theta_4<\pi$. By Lemma \ref{lem:emdedding:strip}, we can admissibly extend $\i$ to a strictly convex pentagon $(A_1A_2A_3A_4A_5)$.
         It then follows from Lemma \ref{lem:pentagon:diagonals} that we can also admissibly extend $\i'$ to a strictly convex pentagon $(A_1'A_2'A_3'A_4'A_5')$ such that
         \begin{equation*}
            F\circ\i(\overline{A_2A_5})=\i'(\overline{A'_2A'_5}),
            ~F\circ\i(\overline{A_3A_5})=\i'(\overline{A'_3A'_5}).
         \end{equation*}
         By Lemma \ref{lem:pentagon:side2}, there exists
         $$\overline{A_mA_n}\in\{\overline{A_1A_5},\overline{A_4A_5}, \overline{A_2A_3}\}$$
         such that $ F\circ\i(\overline{A_mA_n})=\i'(\overline{A'_mA'_n})$.
         \begin{itemize}
           \item If $ F\circ\i(\overline{A_2A_3})=\i'(\overline{A'_2A'_3})$. The algorithm terminates.
           \item Otherwise, we may assume that
           $ F\circ\i(\overline{A_1A_5})=\i'(\overline{A'_1A'_5})$. Let us relabel the quadrilaterals $(A_1A_2A_3A_5)$ and $(A'_1A'_2A'_3A'_5)$ by setting $A_{11}=A_1,A_{21}=A_2,A_{31}=A_3,A_{41}=A_5$ and $A'_{11}=A'_1,A'_{21}=A'_2,A'_{31}=A'_3,A'_{41}=A'_5$. Then
           \begin{eqnarray*}
           & \i(\overline{A_{21}A_{31}})=\i(\overline{A_{2}A_{3}}),&
           \i(\overline{A'_{21}A'_{31}})=\i'(\overline{A'_{2}A'_{3}})\\
        &    F\circ\i(\overline{A_{11}A_{31}})= \i'(\overline{A_{11}'A_{31}'}),&
    F\circ\i(\overline{A_{21}A_{41}})= \i'(\overline{A_{21}'A_{41}'}),\\
    &    F\circ\i(\overline{A_{11}A_{21}})= \i'(\overline{A_{11}'A_{21}'}),&
    F\circ\i(\overline{A_{11}A_{41}})= \i'(\overline{A_{11}'A_{41}'}),\\
   & &    F\circ\i(\overline{A_{31}A_{41}})= \i'(\overline{A_{31}'A_{41}'}).
           \end{eqnarray*}
            Let $\theta_{21}$ and $\theta_{31}$ be respectively the interior angles of
           $(A_{11}A_{21}A_{31}A_{41})$ at $A_{21}$ and $A_{31}$. Then
           \begin{eqnarray*}\label{eq:theta23}
           0<\theta_{21}\leq \theta_2,~0<\theta_{31}< \theta_3\\
            0<\theta_{21}+\theta_{31}<\theta_2+\theta_3<2\pi.
           \end{eqnarray*}
            If $\theta_{21}+\theta_{31}\leq \pi$, the algorithm terminates. Otherwise, we repeat the construction above for $(A_{11}A_{21}A_{31}A_{41})$.
         \end{itemize}
  \end{quotation}

  We now show that the algorithm will terminate after finitely many steps. Suppose to the contrary that the algorithm will never stop. Then in each step, we construct two admissible maps $\i:(A_{1i}A_{2i}A_{3i}A_{4i})\to(X,\omg;\Sigma)$ and
  $\i':(A'_{1i}A'_{2i}A'_{3i}A'_{4i})\to(X',\omg';\Sigma')$  such that
  \begin{eqnarray*}
           & \i(\overline{A_{2i}A_{3i}})=\i(\overline{A_{2}A_{3}}),&
           \i(\overline{A'_{2i}A'_{3i}})=\i'(\overline{A'_{2}A'_{3}})\\
        &    F\circ\i(\overline{A_{1i}A_{3i}})= \i'(\overline{A_{1i}'A_{3i}'}),&
    F\circ\i(\overline{A_{2i}A_{4i}})= \i'(\overline{A_{2i}'A_{4i}'}),\\
    &    F\circ\i(\overline{A_{1i}A_{2i}})= \i'(\overline{A_{1i}'A_{2i}'}),&
    F\circ\i(\overline{A_{1i}A_{4i}})= \i'(\overline{A_{1i}'A_{4i}'}),\\
   & &    F\circ\i(\overline{A_{3i}A_{4i}})= \i'(\overline{A_{3i}'A_{4i}'}).
           \end{eqnarray*}
           and
          \begin{eqnarray}
           0<\theta_{2,i+1}\leq \theta_{2i}\leq\theta_2,~0<\theta_{3,i+1}\leq \theta_{3i}\leq \theta_2 \label{eq:23i1}\\
            \pi<\theta_{2,i+1}+\theta_{3,i+1}<\theta_{2i}+\theta_{3i}<2\pi,\label{eq:23i2}
           \end{eqnarray}
           where for the third  inequality, we use  the observation that either $\theta_{2,i+1}< \theta_{2i}$ or $\theta_{3,i+1}< \theta_{3i}$ for each $i\geq1$.
           Consequently, either $\{\theta_{2i}\}_{i\geq1}$ or $\{\theta_{3i}\}_{i\geq1}$ has a strictly decreasing subsequence, say  $\{\theta_{2i}\}_{i\geq1}$. For convenience, we also denote by $\{\theta_{2i}\}_{i\geq1}$ the decreasing subsequence. In particular, the corresponding saddle connections $\{\i(\overline{A_{1i}A_{2i}})\}_{i\geq1}$ are pairwise different. On the other hand, it follows from (\ref{eq:23i1}) and (\ref{eq:23i2}) that
          \begin{equation}\label{eq:theta23i}
            \pi-\theta_3<\theta_{2i}\leq \theta_2,~
            \pi-\theta_2<\theta_{3i}\leq \theta_3.
          \end{equation}
          Therefore,
          $$\i( |\overline{A_{1i}A_{2i}}|)= |\overline{A_{1i}A_{2i}}|=
           \frac{2\mathbf{Area}(A_{1i}A_{2i}A_{3i})}
           { |\overline{A_{2i}A_{3i}}|\sin\theta_{2i}}
           \leq  \frac{2\mathbf{Area}(X,\omg;\Sigma)}
           {|\overline{A_{2}A_{3}}|\min\{\sin\theta_2,\sin\theta_3\}}:=\mathbf{L}.$$
         Recall that there are finitely many saddle connections on $(X,\omg;\Sigma)$ with length at most $\mathbf{L}$. This is a contradiction which proves the lemma.

\end{proof}

\begin{proof}[Proof of Theorem \ref{thm:pentagon:preserving}]
By Lemma \ref{lem:pentagon:diagonals}, there exists an admissible map  $$\i': (A'_1A_2'A_3'A_4'A'_5)\to(X',\omg';\Sigma')$$  such that $F\circ\i(\overline{A_iA_j})=\i'(\overline{A_i'A_j'})$  for all diagonals $\overline{A_iA_j}$.
By Lemma \ref{lem:pentagon:side2}, we may assume that $F\circ(\overline{A_1A_2})=\i'(\overline{A_1'A_2'})$ and $F\circ(\overline{A_2A_3})=\i'(\overline{A_2'A_3'})$. Consider the quadrilateral $(A_1A_2A_3A_4)$, it follows from Lemma \ref{cor:3to4} that $F\circ\i(\overline{A_3A_4})=\i'(\overline{A_3'A_4'})$. Similarly, by considering
$(A_1A_2A_3A_5)$ and $(A_2A_3A_4A_5)$, we see that
$F\circ\i(\overline{A_4A_5})=\i'(\overline{A_4'A_5'})$ and $F\circ\i(\overline{A_5A_1})=\i'(\overline{A_5'A_1'})$.
\end{proof}


\section{Homeomorphism}\label{sec:homeo}
The goal of this section is to prove the following theorem.
\begin{theorem}\label{thm:homeo}
  Every isomorphism $F:\ms(X,\omg;\Sigma)\to\ms(X',\omg';\Sigma')$  is induced by a homeomorphism $f: (X,\omg;\Sigma)\to (X',\omg';\Sigma')$, such that
  \begin{enumerate}[(i)]
    \item $f(\Sigma)=\Sigma'$, and
    \item  for any saddle connection $\gm$ on $(X,\omg;\Sigma)$, $f(\gm)$ is isotopic to $F(\gm)$.
  \end{enumerate}
  Moreover, if $(X,\omg;\Sigma)$ is not a torus with one marked point, then such a homeomorphism is unique up to isotopy. If $(X,\omg;\Sigma)$ is a torus with one marked point, then there are two such homeomorphisms up to isotopy.
\end{theorem}

\subsection{Triangulation graphs}
\begin{definition}
  Two triangulations $\Gm_1$ and $\Gm_2$ of $(X,\omg;\Sigma)$  differ by an \emph{elementary move} if
  \begin{itemize}
    \item there exist $\bt_1\in\Gm_1$ and $\bt_2\in\Gm_2$ such that $\Gm_1\backslash\bt_1=\Gm_2\backslash\bt_2$, and
    \item there exist $\gm_1,\gm_2,\gm_3,\gm_4\in
        \Gm_1\backslash\{\bt_1\}=\Gm_2\backslash\{\bt_2\}$ such that they bound a strictly convex quadrilateral on $(X,\omg;\Sigma)$ which contains $\bt_1,\bt_2$ as diagonals.
  \end{itemize}
\end{definition}
\begin{definition}
  The \emph{triangulation graph} of $(X,\omg;\Sigma)$, denoted by $\mT(X,\omg;\Sigma)$, is a graph whose vertices are triangulations of $(X,\omg;\Sigma)$, and whose edges are pairs of triangulations which differ by an elementary move.
\end{definition}

\begin{proposition}[\cite{BS,ILTC,Ngu3,Tah}]
\label{prop:tri:connect}
  For any ~half-translation~ surface with marked points $(X,\omg;\Sigma)$, the triangulation graph $\mT(X,\omega)$ is connected.
\end{proposition}
\remark The above proposition holds for general flat surfaces (simplicial surfaces) (see \cite[Proposition 11,12]{BS},\cite[Theorem 1]{ILTC},\cite[Theorem 1.5]{Tah}). Nguyen (\cite[Theorem 6.2]{Ngu3}) provides an elementary proof for the case of half-translation~ surfaces.

\subsection{Orientation consistency}
By Theorem
\ref{thm:triangle}, we know that if the saddle connections $\gm_1,\gm_2,\gm_3 $  bound a triangle $\Dt$ on $ (X,\omg;\Sigma)$, their images $F(\gm_1),F(\gm_2),F(\gm_3)$  also bound a  triangle on $(X',\omg';\Sigma')$, which is denoted by $\Dt'$. This correspondence induces an affine homeomorphism between $\Dt$ and $\Dt'$, which is called the \emph{$F$-induced affine homeomorphism} and denoted by $f_\Dt$. Our goal is to  ``glue" these $F$-induced affine homeomorphisms  between triangles according to some triangulation of $(X,\omg;\Sigma)$ to obtain a globally well defined homeomorphism  from $(X,\omg;\Sigma)$ to $(X',\omg';\Sigma')$. To do this,  we need to clarify the orientation consistency among affine homeomorphisms between triangles.

\begin{definition}
  Two triangles $\Delta_1$ and $\Delta_2$ on $(X,\omg;\Sigma)$   are called \emph{coconvex} if there exists an admissible map $\i:(A_1A_2A_3A_4)\to(X,\omg;\Sigma)$ such that
  \begin{itemize}
    \item $(A_1A_2A_3A_4)$ is  strictly convex at each vertex;
    \item  both $\Delta_1$ and $\Delta_2$ are contained in the image of $(A_1A_2A_3A_4)$ by $\i$.
  \end{itemize}
\end{definition}

\begin{lemma}\label{lem:quadrilateral:consistence}
  Let $\Delta$ and $\hat{\Delta}$ be two coconvex triangles on $(X,\omg;\Sigma)$. Suppose that the $F$-induced affine homeomorphism for $\Delta$ is orientation preserving. Then the $F$-induced affine homeomorphism  for $\hat{\Delta}$ is also orientation preserving.
\end{lemma}

\begin{proof}
   Let $\i: (A_1A_2A_3A_4)\to (X,\omg;\Sigma)$ be an admissible  map whose image contains both $\Delta$ and $\hat{\Delta}$. In particular, $(A_1A_2A_3A_4)$  is strictly convex at each vertex.  Suppose that $\i(A_1A_2A_4)=\Delta$.
  By Lemma \ref{lem:quadri:diagonals} and Theorem \ref{thm:triangle} there exists an admissible map $\i': (A_1'A_2'A_3'A_4')\to (X',\omg';\Sigma')$, where $ (A_1'A_2'A_3'A_4')$ is a strictly convex quadrilateral, such that  $F\circ\i(\overline{A_iA_j})=\i'(\overline{A_i'A_j'})$ for all $i\neq j$. Suppose that the vertices of $(A_1A_2A_3A_4)$ are labeled in the counterclockwise order. Then the vertices of $(A_1A_2A_4)$ is also labeled in the counterclockwise order. Since  the $F$-induced affine homeomorphism on $\i(A_1A_2A_4)$ is orientation preserving, we see that the vertices of $(A_1'A_2'A_4')$ is also  labeled in the counterclockwise order. Therefore, the vertices  of $(A_1'A_2'A_3'A_4')$ are  labeled in the counterclockwise order. (Otherwise,  the vertices of $(A_1'A_2'A_4')$ would be   labeled in the clockwise order.) As a consequence, the $F$-induced affine homeomorphism on the image of any triangle in $(A_1A_2A_3A_4)$ is also orientation preserving.
  In particular,  the $F$-induced affine homeomorphism  for $\hat{\Delta}$ is also orientation preserving.
\end{proof}

\begin{lemma}\label{lem:trianglescoconvex}
  For any two triangles $\Delta_1$ and ${\Delta}_2$ on $(X,\omg;\Sigma)$, there exists a sequence of triangles $\Delta_{0}=\Delta,\Delta_{1},\cdots,\Delta_{m+1}={\Delta}_2$ such that $\Delta_k$ and $\Delta_{k+1}$ are coconvex for each $0\leq k\leq m$.
\end{lemma}

\begin{proof}
   Notice that each triangle is a connected  component of $(X,\omg;\Sigma)\backslash\Gm$ for some triangulation $\Gm$. To prove the lemma,
   by Proposition \ref{prop:tri:connect}, it suffices to consider the case that
   $$ \Delta_1\subset (X,\omg;\Sigma)\backslash\Gm_1,~
    {\Delta}_2\subset (X,\omg;\Sigma)\backslash{\Gm}_2, $$
    where $\Gm_1$ and ${\Gm}_2$ are two triangulations differing by an elementary move.

    By definition, there exists an admissible map $\i:(A_1A_2A_3A_4)\to(X,\omg;\Sigma)$, where $(A_1A_2A_3A_4)$ is a strictly convex quadrilateral, such that
    $$ \i(\overline{A_1A_3})\in\Gm_1, ~\i(\overline{A_2A_4})\in{\Gm}_2,
    ~\Gm_1\backslash\{ \i(\overline{A_1A_3})\}= {\Gm}_2\backslash\{ \i(\overline{A_2A_4})\}.$$
    Let $\hat{\Delta}_1=\i(A_1A_2A_3)$ and $\hat{\Delta}_2=\i(\overline{A_2A_3A_4})$. Then $\hat{\Delta}_1$ and $\hat{\Delta}_2$ are coconvex.

    Consider the triangles in  $(X,\omg;\Sigma)\backslash\Gm_1$  and $(X,\omg;\Sigma)\backslash{\Gm_2}$. There exist  triangles $\tilde{\Delta}_{1,j}$ ($0\leq j\leq m_1$) and  $\tilde{\Delta}_{2,j}$ ($0\leq j\leq m_2$),  such that
    \begin{itemize}
      \item $\tilde{\Delta}_{1,0}={\Delta}_1$, $\tilde{\Delta}_{1,m_1}=\hat{\Delta}_1$,
      $\tilde{\Delta}_{2,0}=\hat{\Delta}_2$, $\tilde{\Delta}_{2,m_2}={\Delta}_2$;
      \item $\tilde{\Delta}_{i,j},\tilde{\Delta}_{i,j+1}\subset(X,\omg;\Sigma)\backslash\Gm_i $, $\forall i=1,2, \forall 0\leq j\leq m_i-1$;
      \item $\tilde{\Delta}_{i,j}$ and $ \tilde{\Delta}_{i,j+1}$ share a common boundary, $\forall i=1,2, \forall 0\leq j\leq m_i-1$.
    \end{itemize}

 Next, consider the pair $(\tilde{\Delta}_{i,j},\tilde{\Delta}_{i,j+1})$. By Lemma \ref{lem:coconvex}, there exists a sequence of triangles $\tilde{\Delta}_{i,j,0}=\tilde{\Delta}_{i,j}, \tilde{\Delta}_{i,j,1}, \cdots ,
\tilde{\Delta}_{i,j,n_{ij}}=\tilde{\Delta}_{i,j+1}$, such that each pair of adjacent triangles are coconvex.

 Now,  let us replace each pair $(\tilde{\Delta}_{i,j},\tilde{\Delta}_{i,j+1})$ by the sequence $\tilde{\Delta}_{i,j,0}=\tilde{\Delta}_{i,j}$, $\tilde{\Delta}_{i,j,1}$, $\cdots ,
\tilde{\Delta}_{i,j,n_{ij}}=\tilde{\Delta}_{i,j+1}$. Then for the  new sequence which starts at $\Delta_1$ and ends at ${\Delta}_2$, any adjacent pair of triangles are coconvex.

\end{proof}

 Combining  Lemma \ref{lem:quadrilateral:consistence} and Lemma \ref{lem:trianglescoconvex}, we have the following proposition.
\begin{proposition}\label{prop:orientation:consist}
 Let  $F:\ms(X,\omg;\Sigma)\to\ms(X',\omg';\Sigma')$ be an isomorphism.
 Then either
  \begin{enumerate}[(i)]
    \item for every triangle on $(X,\omg;\Sigma)$, the induced affine homeomorphism between triangles is orientation preserving, or
    \item for every triangle on $(X,\omg;\Sigma)$, the induced affine homeomorphism between triangles is orientation reversing.
  \end{enumerate}
\end{proposition}

\begin{corollary}\label{cor:triangulation:homeo}
  Let  $F:\ms(X,\omg;\Sigma)\to\ms(X',\omg';\Sigma')$ be an isomorphism. Then any triangulation $\Gm$ of $(X,\omg;\Sigma)$ induces a homeomorphism $f_\Gm$ between $(X,\omg;\Sigma)$ and $(X',\omg';\Sigma')$,  such that
  \begin{enumerate}
    \item $f_\Gm(\Sigma)=\Sigma'$, and
    \item  for any saddle connection $\gm\in\Gm$, $f_\Gm(\gm)=F(\gm)$.
  \end{enumerate}
\end{corollary}
\begin{proof}
  It follows from Theorem \ref{thm:triangle} and Proposition \ref{prop:orientation:consist}.
\end{proof}
The homeomorphism $f_\Gm$ obtained in Corollary \ref{cor:triangulation:homeo} is called the {$F$-\textit{induced homeomorphism} with respect to $\Gm$.
In the following, we prove that the isotopy class of $f_\Gm$ relative to $\Sigma$ and $\Sigma'$ is independent of the choices of triangulations.

\begin{proposition}\label{prop:consist:triangulation}
Let  $F:\ms(X,\omg;\Sigma)\to\ms(X',\omg';\Sigma')$ be an isomorphism.
For any two triangulations $\Gm_1$ and $\Gm_2$, the $F$-induced  homeomorphisms $f_{\Gm_1}$ and $f_{\Gm_2}$ are isotopic.
\end{proposition}
\begin{proof}
 By Proposition \ref{prop:tri:connect}, it suffices to prove it for the case that $\Gm_1$ and $\Gm_2$ differ by an elementary move. Let $\Gm_1=\{\alpha_1,\gm_2,\gm_3,\gm_4,\gm_5,\cdots,\gm_k\}$ and $\Gm_2=\{\bt_1,\gm_2,\gm_3,\gm_4,\gm_5,\cdots,\gm_k\}$ such that
$\gm_2,\gm_3,\gm_4,\gm_5$ bound a strictly convex quadrilateral $Q$ on $(X,\omg;\Sigma)$ whose diagonals are $\alpha_1,\bt_1$. Correspondingly,  $F(\gm_2),F(\gm_3)$, $F(\gm_4),F(\gm_5)$ bound a strictly convex quadrilateral $Q'$ on $(X',\omg';\Sigma')$ whose diagonals are $F(\alpha_1),F(\bt_1)$. By construction, $f_{\Gm_1}|_{X_1\backslash Q}=f_{\Gm_2}|_{X_1\backslash Q}$. Notice that $\alpha_1$ and $\beta_1$ divide $Q$ into four triangles. Let $ f_Q$ be the piecewisely affine map from $Q$ to $Q'$ whose restriction to each of these four triangles is affine. Let $f_{12}:(X,\omg;\Sigma)\to(X',\omg';\Sigma')$ be a homeomorphism such that $f_{12}|_{X_1\backslash Q}= f_{\Gm_1}|_{X_1\backslash Q}=f_{\Gm_2}|_{X_1\backslash Q}$ and $f_{12}|_Q=f_Q$. Then both $f_{\Gm_1}$  and $f_{\Gm_2}$ are isotopic to $f_{12}$. Therefore, $f_{\Gm_1}$ and $f_{\Gm_2}$ are isotopic.

\end{proof}

\begin{proof}[Proof of Theorem \ref{thm:homeo}.]
  If $(X,\omg;\Sigma)$ is not a torus with one marked point, they every triple of saddle connections bound at most one triangle. Then the  theorem follows from Corollary \ref{cor:triangulation:homeo} and Proposition \ref{prop:consist:triangulation}.

  If $(X,\omg;\Sigma)$ is a tours with one marked point. Then every triangulation $\Gm$ consists of three saddle connections, which bound two triangles on $(X,\omg;\Sigma)$. Therefore, the $F$-induced homeomorphisms with respect to $\Gm$ has two choices, which results in two isotopy classes of homeomorphisms satisfying the condition in the theorem.
\end{proof}

\section{Affine Homeomorphism}\label{sec:affine}

Our goal in this section is to prove the following proposition, which is the last piece for proving  the main theorem.

\begin{proposition}\label{prop:affine}
  Let  $f:(X,\omg;\Sigma)\to(X',\omg';\Sigma')$  be a homeomorphism which induces an isomorphism $f_*:\ms(X,\omg;\Sigma)\to\ms(X',\omg';\Sigma')$. Then $f$ is isotopic to an affine homeomorphism.
\end{proposition}

\begin{proof}

  Recall that a simple closed curve on a half-translation surface with marked points is called a cylinder  curve if it is isotopic (relative to marked points) to the core curve of some cylinder.

  By Theorem \ref{thm:DLR}, it suffices to prove that for each simple closed curve $\alpha$, $\alpha$ is a cylinder curve on $(X,\omg;\Sigma)$ if and only if $f(\alpha)$ is a cylinder curve on $(X',\omg';\Sigma')$. Since $f_*$ is an isomorphism,  it suffices to prove that if
  $\alpha$ is a cylinder curve on $(X,\omg;\Sigma)$, then $f(\alpha)$ is a cylinder curve on $(X',\omg';\Sigma')$.

  Let $\C$ be an  cylinder on $(X,\omg;\Sigma)$ with core curve $\alpha$. Let
  $$\i:(Q_1Q_2\cdots Q_n Q_{n+1}\cdots Q_m)\to (X,\omg;\Sigma)$$
   be an admissible map such that
   \begin{itemize}
     \item $\i(\overline{Q_nQ_{n+1}})=\i(\overline{Q_mQ_1})$ is an interior saddle connection of $\C$;
     \item $\i(\overline{Q_iQ_{i+1}})$ is a boundary saddle connection for $1\leq i\leq n-1$ and $n+1\leq i\leq m-1$.
   \end{itemize}

  If $\C$ is a simple cylinder, it then follows from Proposition \ref{lem:simplecylinder} that  $f(\alpha)$ is a cylinder  curve on $(X',\omg';\Sigma')$.

  If $\C$ is not a simple cylinder, consider a triangulation of $$(Q_1Q_2\cdots Q_n Q_{n+1}\cdots Q_m).$$ By Theorem \ref{thm:triangle}, there exists an admissible map $$\i':(Q'_1Q'_2\cdots Q'_n Q'_{n+1}\cdots Q'_m)\to (X',\omg';\Sigma')$$ such that $f_*(\i(\overline{Q_iQ_j}))=\i'(\overline{Q_i'Q_j'})$ for all sides and diagonals $\overline{Q_iQ_j}$.

  To show that $f(\alpha)$ is a cylinder curve on $(X',\omg';\Sigma')$, it is equivalent to show that   $\overline{Q_i'Q_{i+1}'}$ and $\overline{Q_j'Q_{j+1}'}$ are parallel for any $1\leq i\leq n-1$ and $n+1\leq j\leq m-1$.

  Consider the quadrilateral $(Q'_1Q'_2Q'_{n+1}Q'_{n+2}).$ Let $\theta_1',\theta_2'$, $\theta_{n+1}',\theta_{n+2}'$ be the interior angles at $Q_1',Q_2',Q_{n+1}'$ and $Q_{n+2}'$ respectively. If $\overline{Q_1'Q_2'}$ and $\overline{Q_{n+1}'Q_{n+2}'}$ are not parallel, then either $\theta_2'+\theta_{n+1}'>\pi$ or $\theta_{n+2}'+\theta_1'>\pi$. Without loss of generality, we may assume that $\theta_2'+\theta_{n+1}'>\pi$. It then follows from Lemma \ref{lem:emdedding:strip} that we can admissibly extend $\i'|_{(Q_1'Q_2'Q_{n+1}'Q_{n+2}')}$ to a strictly convex pentagon $(Q_1'Q_2'Q_{n+1}'Q_{n+2}'B')$. Correspondingly, by Theorem \ref{thm:pentagon:preserving}, we can also admissibly extend  $\i|_{(Q_1Q_2Q_{n+1}Q_{n+2})}$ to a strictly convex pentagon $(Q_1Q_2Q_{n+1}Q_{n+2}B)$, which contradicts to the assumption that
  $\i(Q_1Q_2Q_{n+1}Q_{n+2})$ is contained in the cylinder $\C$ whose boundary is parallel to $\i(\overline{Q_1Q_2})$. Therefore,  $\overline{Q_1'Q_2'} $ and $\overline{Q_{n+1}'Q_{n+2}'}$ are parallel. Similarly, we can show that $\overline{Q_i'Q_{i+1}'}$ and $\overline{Q_j'Q_{j+1}'}$ are parallel for any $1\leq i\leq n-1$ and $n+1\leq j\leq m-1$.

\end{proof}

\begin{proof}[Proof of Theorem \ref{thm:main1}]
  By Theorem \ref{thm:homeo}, there exists a homeomorphism $f:(X,\omg;\Sigma)\to(X',\omg';\Sigma')$ such that for any saddle connection $\gm$ on $(X,\omg;\Sigma)$, $f(\gm)$ is isotopic to $F(\gm)$. It then follows from Proposition \ref{prop:affine} that $f$ is isopotic to an affine homeomorphism $\tilde{f}:(X,\omg;\Sigma)\to (X',\omg';\Sigma')$.
\end{proof}
\begin{proof}[Proof of Theorem \ref{thm:main2}]
  Let $\mathrm{Aut}(\ms(X,\omg;\Sigma))$ be the automorphism group of $\ms(X,\omg;\Sigma)$. Let $\mathrm{Aff}(X,\omg;\Sigma)$ be the group of  affine self-homeomorphisms  of $(X,\omg;\Sigma)$.
  Since each affine self-homeomorphism of $(X,\omg;\Sigma)$ induces an automorphism of $\ms(X,\omg;\Sigma)$, there exists a natural group homomorphism $$\mathbb{F}:  \mathrm{Aff}(X,\omg;\Sigma)\to\mathrm{Aut}(\ms(X,\omg;\Sigma).$$
  By Theorem \ref{thm:main1}, $\mathbb{F}$ is surjective. Moreover, it follows from  Theorem \ref{thm:homeo} that if $(X,\omg;\Sigma)$ is a torus with one marked point, then $\mathbb{F}$ is two-to-one. Otherwise, $\mathbb{F}$ is injective.

\end{proof}

\section{ Quotient graph}\label{sec:quotient}

\begin{figure}
  \centering
 \begin{tikzpicture}
   \path[fill,black]
   (0,1)coordinate(a1) circle(0.05)
   (1,1)coordinate(a2) circle(0.05)
   (2.5,1)coordinate(a3) circle(0.05)
   (4.5,1)coordinate(a4) circle(0.05)
   (7,1)coordinate(a5) circle(0.05)
   (0,0)coordinate(b1) circle(0.05)
   (2,0)coordinate(b2) circle(0.05)
   (3,0)coordinate(b3) circle(0.05)
   (5.5,0)coordinate(b4) circle(0.05)
   (7,0)coordinate(b5) circle(0.05);
   \draw (a1)--(a2)--(a3)--(a4)--(a5)--(b5)
   --(b4)--(b3)--(b2)--(b1)--(a1);
   \path[fill,red](a1)--(b2)--(b3)--(a2)--cycle;
   \draw(0.5,1)node[above]{$\gm^+_1$}
   (1.8,1)node[above]{$\gm^+_2$}
   (3.5,1)node[above]{$\gm^+_3$}
   (5.7,1)node[above]{$\gm^+_4$}
   (1,0)node[below]{$\gm^-_3$}
    (2.5,0)node[below]{$\gm^-_1$}
     (4.2,0)node[below]{$\gm^-_4$}
      (6.2,0)node[below]{$\gm^-_2$};
 \end{tikzpicture}
  \caption{\small{The red region represents a simple cylinder determined by $\gm^{\pm}_1$ and the saddle connection connecting the left endpoints of $\gm^{\pm}_1$.} }\label{fig:simp}
\end{figure}
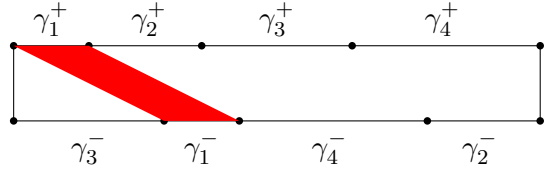

The goal of this section is to prove Theorem \ref{thm:quotient}.  Recall that for two vectors $\vec{a}=(x_1,y_1) $ and $\vec{b}=(x_2,y_2)$ on $\R^2$,  $\vec{a}\wedge \vec{b}:=x_1y_2-x_2y_1$. We start with the following lemma.
\begin{lemma}\label{lem:disjoint}
  If $(X,\omg;\Sigma)$ is a translation surface which is not a torus with one marked point, then there are saddle connections $\alpha, \bt_1,\bt_2,\cdots$, such that $\bt_i$ is disjoint from $\alpha$ for all $i\geq1$, and that $\lim_{i\to\infty}|\int_{\alpha}{\omg} \wedge\int_{\beta_i}{\omg}|=\infty.$
\end{lemma}
\begin{proof}
  It follows from \cite[Theorem 2]{Mas} that $(X,\omg;\Sigma)$ has infinitely many cylinders.

  First, we claim that there exists a cylinder $C_0$ such that the closure is a proper subset of $(X,\omg;\Sigma)$. Indeed, let $C\subset(X,\omg;\Sigma)$ be a cylinder. Without loss generality, we may assume that $C$ is horizontal. If $(X,\omg;\Sigma)\neq \overline{C}$, the claim follows. If $(X,\omg;\Sigma)= \overline{C}$,  then for every saddle connection $\dt^+$ in the upper boundary component of $C$, there is a corresponding saddle connection $\dt^-$ in the lower boundary component, such that $\int_{\gm^+}{\omg}=\int_{\gm^-}{\omg}$ (see Figure \ref{fig:simp}). Let $(\gm^+_1,\gm^-_1)$ be such a pair, they determine a simple cylinder $C_1$ as illustrated in Figure \ref{fig:simp}. By assumption, $(X,\omg;\Sigma)$ is not a torus with one marked point, each boundary component of $C$ contains at least two saddle connections. In particular, $\overline{C_1}$ is a proper subset of $\overline{C}=(X,\omg;\Sigma)$.

  Next, let $\alpha$ be a non-horizontal saddle connection on $(X,\omg;\Sigma)\backslash \overline{C_0}$, let $\{\bt_i\}_{i\geq1}$ be a sequence of interior saddle connections of $C_0$. Then  $\bt_i$ is disjoint from $\alpha$ for all $i\geq1$, and that $\lim_{i\to\infty}|\int_{\alpha}{\omg} \wedge\int_{\beta_i}{\omg}|=\infty.$

\end{proof}

\begin{proof}[Proof of Theorem \ref{thm:quotient}]
 (i)  Notice that every edge in the saddle connection graph is represented by a pair of disjoint saddle connections.

 If $(X,\omg;\Sigma)$ is a torus with one marked point, the group $\Aff ^+(X,\omg;\Sigma)$ of orientation-preserving affine homeomorphisms of $(X,\omg;\Sigma)$ is isomorphic to $SL(2,\Z)$, and the set of pairs of
 disjoint saddle connections has one $\Aff^+(X,\omg;\Sigma)$-orbit. In particular, $\mathcal{G}(X,\omg;\Sigma)$ has one vertex and  one edge.

 If $(X,\omg;\Sigma)$ is not a torus with one marked point, let $\{\alpha_1,\beta_1\}$, $\{\alpha_2,\beta_2\}$  be two pairs of  non-parallel, disjoint saddle connections on $(X,\omg;\Sigma)$. They represent two edges $e_1,e_2$ of $\ms(X,\omg;\Sigma)$. Suppose that  there is an automorphism $F$ of $\ms(X,\omg;\Sigma)$ such that
 $F(e_1)=e_2$. By Theorem \ref{thm:main1}, there is an affine homeomorphism $f$ of $(X,\omg;\Sigma)$ such that $f(\{\alpha_1,\beta_1\})=\{\alpha_2,\beta_2\}$. Therefore, $|\int_{\alpha_1}\omg\wedge\int_{\beta_1}\omg|=
 |\int_{\alpha_2}\omg\wedge\int_{\beta_2}\omg|\neq 0$. On the other hand, by Lemma \ref{lem:disjoint}, the set \[\{|\int_{\alpha}\omg\wedge\int_{\beta}\omg|: \alpha,\bt \text{ are disjoint saddle connections}\}\]
  is an infinite set.   As a consequence,  $\mathcal{G}(X,\omg;\Sigma)$ has infinitely many edges.

 \vskip 5pt
 (ii)  If $(X,\omg;\Sigma)$ is Veech surface, it follows from \cite[Theorem 6.8]{Vor} (see also \cite[Theorem 1.3]{SW}) that the set of triangles on $(X,\omg;\Sigma)$ has finitely many $\Aff(X,\omg;\Sigma)$-orbits, where $\Aff(X,\omg;\Sigma)$ is the group of affine homeomorphisms of $(X,\omg;\Sigma)$. This implies that the set of saddle connections has finitely many  $\Aff(X,\omg;\Sigma)$-orbits, since each saddle connection is contained in at least one triangle. In particular, $\mathcal{G}(X,\omg;\Sigma)$ has finitely many vertices.

 \end{proof}

\section{Questions}\label{sec:questions}

In this section, we propose two questions.
The first question concerns  Theorem \ref{thm:quotient}.
\begin{question}\label{q2}
  (i) Characterize those half-translation surfaces whose quotient graph have finitely many edges.
 (ii)  Is it true that the quotient graph has finitely many vertices if and only if the underlying half-translation surface is a Veech surface?
\end{question}

 \remark Let $\mathbb{T}(X,\omg;\Sigma)$ be the spine tree defined by Smillie-Weiss (see \cite[\S4]{SW} for the definition). Suppose that $\mathcal{G}(X,\omg;\Sigma)$ has finite vertices. Then the set of directions of saddle connections  has finite $\Aff(X,\omg;\Sigma)$-orbits, which implies that the set of components of $\mathbb H^2\backslash\mathbb{T}(X,\omg;\Sigma)$ has finite $\Aff(X,\omg;\Sigma)$-orbits. To prove that $(X,\omg;\Sigma)$ is a Veech  surface, it suffices to prove that every  component  of  $\mathbb H^2\backslash\mathbb{T}(X,\omg;\Sigma)$ has finite quotient area by $\Aff(X,\omg;\Sigma)$. This is equivalent to show that every saddle connection has non-trivial stabilizers in $\Aff(X,\omg;\Sigma)$.

Irmak-McCarthy (\cite{IM}) proved that every injective simplicial map from an arc graph to itself is induced by some self-homeomorphism of the underlying surface (see also \cite{Ara,Irm1,Irm2}). We may ask a similar question for the saddle connection graph.
\begin{question}\label{q3}
  Let $(X,\omg;\Sigma)$ be a ~half-translation~ surface with marked points. Is it true that every injective simplicial map $F:\ms(X,
  \omg,\Sigma)\to\ms(X,\omg;\Sigma)$ is induced by some affine homeomorphism $f:(X,\omg;\Sigma)\to (X,\omg;\Sigma)$?
\end{question}

 
\end{document}